\documentclass{amsart}
\usepackage{amsmath}
\usepackage{amssymb}
\usepackage{amsthm}
\usepackage{dsfont}
\usepackage{lineno}
\usepackage{subfigure}
\usepackage{graphicx}
\usepackage{multirow}
\usepackage{color}
\usepackage{xspace}
\usepackage{pdfsync}
\usepackage{hyperref} 
\usepackage{algorithmicx}
\usepackage{algpseudocode}
\usepackage{algorithm}
\graphicspath{{../../SweepingFigures/}}

\newcommand{\bq}{\begin{equation}}
\newcommand{\eq}{\end{equation}}
\newcommand{\R}{\mathbb{R}}
\newcommand{\abs}[1]{\left\vert#1\right\vert}

\newcommand{\G}{\mathcal{G}}
\newcommand{\bO}{\mathcal{O}}

\newcommand{\Prop}{\mathcal{P}}
\newcommand{\F}{\mathcal{F}}

\algnewcommand{\LineComment}[1]{\State \(\triangleright\) #1}

\newtheorem{theorem}{Theorem}
\theoremstyle{lemma}

\theoremstyle{remark}

\newtheorem{remark}{Remark}

\begin{document}

\title[Fast Sweeping Methods]{Fast Sweeping Methods for Hyperbolic Systems of Conservation Laws at Steady State}

\author{Bj\"orn Engquist}
\thanks{Bj\"orn Engquist, Department of Mathematics and ICES, The University of Texas at Austin, 1 University Station C1200, Austin, TX 78712
USA (engquist@math.utexas.edu).  This author was partially supported by NSF DMS-1217203.}

\author{Brittany D. Froese}
\thanks{Brittany D. Froese, Department of Mathematics and ICES, The University of Texas at Austin, 1 University Station C1200, Austin, TX 78712
USA (bfroese@math.utexas.edu).  This author was partially supported by an NSERC PDF}

\author{Yen-Hsi Richard Tsai}
\thanks{Yen-Hsi Richard Tsai, Department of Mathematics and ICES, The University of Texas at Austin, 1 University Station C1200, Austin, TX 78712
USA (ytsai@math.utexas.edu).  This author was partially supported by NSF DMS-1217203 and a Moncrief Grand Challenge Award.}

\begin{abstract}
Fast sweeping methods have become a useful tool for computing the solutions of static Hamilton-Jacobi equations.  By adapting the main idea behind these methods, we describe a new approach for computing steady state solutions to systems of conservation laws.  By exploiting the flow of information along characteristics, these fast sweeping methods can compute solutions very efficiently.  Furthermore, the methods capture shocks sharply by directly imposing the Rankine-Hugoniot shock conditions.  We present convergence analysis and numerics for several one- and two-dimensional examples to illustrate the use and advantages of this approach.
\end{abstract}

\date{\today}

\maketitle

\section{Introduction}\label{sec:intro}

The numerical solution of systems of conservation laws,
\bq\label{eq:system}
\begin{cases}
 U_t + \nabla \cdot F(U) = a(U,x), & x\in\Omega, t>0\\
 U = U_0(x), &  x\in\Omega, t=0\\
 B(U,x) = 0, & x\in\partial\Omega, t>0
\end{cases}
\eq
has continued to be an important problem in numerical analysis.  A major challenge associated with this task is the need to compute non-classical solutions~\cite{CourantFriedrichs}, which leads to the need to develop numerical schemes that correctly resolve discontinuities in weak (entropy) solutions.  Several different approaches are now available for resolving shock fronts including front tracking schemes~\cite{Glimm_FrontTracking}, upstream-centered schemes for conservation laws (MUSCL)~\cite{Colella_MUSCL,vanLeer}, central schemes~\cite{JiangTadmor_CentralMultiD,NessyahuTadmor_Central}, essentially non-oscillatory (ENO) schemes~\cite{HartenENO}, and weighted essentially non-oscillatory (WENO) schemes~\cite{LiuOsherChan_WENO,OsherShu_ENO,Shu_ENO}. 

In many applications, it is important to compute the steady state solution of~\eqref{eq:system}, which can be viewed as a particular solution of the boundary value problem
\bq\label{eq:systemSteady}
\begin{cases}
 \nabla \cdot F(U) = a(U,x), & x\in\Omega\\
 B(U,x) = 0, & x\in\partial\Omega.
\end{cases}
\eq
%However, computing steady state solutions remains a computationally expensive problem.

A natural approach to computing steady state solutions is to use an explicit time stepping or pseudo time stepping technique to evolve the system to steady state~\cite{AbgrallMezine_Steady,AbgrallRoe,ChouShuWENO,JiangShuWENO}.  However, the computational efficiency of these schemes is restricted by a CFL condition and the need to evolve the system for a substantial time in order to reach the steady state solution.  In order to substantially improve the efficiency of these computations, it is desirable to develop methods that solve the steady state equations directly instead of through a time-evolution process.

Early work in this direction used Newton's method to solve a discrete version of the boundary value problem~\eqref{eq:systemSteady} using shock tracking techniques~\cite{GustafssonWahlund_SteadyFlow,ShubinStephens_SteadyShock}.
More recently, Newton solvers have been applied to WENO approximations of the steady Euler equations~\cite{HuLiTang_Newton}.  For more general systems, a Gauss-Seidel scheme based on a Lax-Friedrichs discretisation of the steady state equations was described in~\cite{Chen_LFSweeping}.  In~\cite{Shu_WENOHomotopy}, a homotopy approach was introduced to evolve from an initial condition to a steady state solution without the restriction of a CFL condition.

To gain inspiration, we look at some of the techniques that have been developed for solving static Hamilton-Jacobi equations.  Many of the methods commonly used for these equations rely on the fact that information propagates along characteristics.  Fast marching methods~\cite{HelmsenPuckett, Sethian_FastMarching} use fast sorting techniques to order the grid points in a way that allows the solution to be computed with a single pass through the computational domain.  This approach requires strong assumptions on the monotonicity of the solution with respect to the stencil used, which makes it difficult to apply to problems with anisotropy.  Related to this approach are ordered upwind methods~\cite{SethianVladimirsky_OUM}, which use an optimal control formulation to produce a single-pass solution method.  Fast sweeping methods~\cite{KaoOsherTsai_Sweeping,TsaiChenOshwerZhao_Sweeping,ZhaoSweeping} were introduced to avoid the complexity arising from the sorting procedure required by single-pass methods.  Fast sweeping methods, which also make use of the propagation of information along characteristics, involve updating solution values by passing through the computational domain in several pre-determined sweeping directions.  This typically leads to algorithms with linear computational complexity.

We introduce a new computational approach for steady state conservation laws that is based on the spirit of the fast sweeping methods.
Our fast sweeping approach has a number of advantages.  The most immediate advantage is the low computational cost, which is optimal, $\bO(N)$, where $N$ is the number of unknowns in the representation of the solution.  Secondly, the methods compute shocks sharply by directly imposing the Rankine-Hugoniot shock conditions, together with appropriate entropy conditions.  The methods are also flexible in the sense that they can be combined with any reasonable numerical approximation of the flux functions; in fact, in many cases it is possible to obtain correct shock locations using non-conservative schemes.  In particular, this allows the easy use of higher order approximation schemes.  In some situations, a system of conservation laws (with reasonable boundary conditions) will not have a unique steady state solution; our fast sweeping methods can be used to compute multiple steady states when necessary~(\S\ref{sec:isentropic2}).  Even entropy shocks that are unstable when embedded in time evolution processes, and therefore cannot be computed using time-stepping based methods, are accessible to our method~(\S\ref{sec:isentropic2}).  Finally, we note that different types of boundary conditions are appropriate in different settings, and some components of the solution vector may not be explicitly given on the entire boundary.  However, our methods are powerful enough to solve steady state problems that are well-posed with these ``incomplete'' boundary conditions; we do not require the problem to be overdetermined through specification of all solution components at the boundary~(\S\ref{sec:nozzle}).

The details of the methods will be given in the following sections.  Here we simply point out the two main steps.
\begin{enumerate}
\item Solution branches are generated by means of an update formula that is used to update the solution along different sweeping directions.  In one dimension, for example, one solution branch is obtained by sweeping through the domain from left to right, and another is obtained by sweeping from right to left.  In higher dimensions, more sweeping directions are typically employed.  When an incomplete set of boundary conditions is given, unknown components of the solution vector at the boundary must be supplied.  These are determined via an iteration between steps~(1) and~(2).
\item A selection principle is used to determine which solution branch is active at each point.  
%For Hamilton-Jacobi equations, this selection principle is typically determined by a monotonicity requirement on the numerical Hamiltonian.  
In the case of nonlinear conservation laws, the Rankine-Hugoniot conditions for a stationary shock provide a set of equations that determines the shock location and any missing boundary conditions.  The numerical algorithm for solving this set of equations guides the iteration.  Entropy conditions are also applied to verify the validity of the shock.
\end{enumerate}

\section{Background}\label{sec:background}
Before we provide the details of our fast sweeping method, we provide some background material that will inform the approach taken in this work.

\subsection{Sweeping Methods}\label{sec:sweeping}
The methods we describe here are motivated by the fast sweeping methods for the solution of static Hamilton-Jacobi equations.  Fast sweeping methods rely on the fact that boundary data will propagate into the domain along characteristic directions.  
%At its essence, the procedure involves two main steps:

We illustrate the basic principles of fast sweeping methods by considering the simple one-dimensional Hamilton-Jacobi equation
\bq\label{eq:eikonal}
\begin{cases}
({u_x})^2 = u, & 0<x<1\\
u = 1, & x= 0, 1.
\end{cases}
\eq

We can make a few observations about this boundary-value problem.  First of all, no smooth (that is, $C^1$) solution exists; instead, we are interested in the viscosity solution of the equation~\cite{CrandallLions_ViscHJ}.  
%We can also note in passing that the viscosity solution of the equation is simply the distance to the boundary.  To ensure that we are computing the correct weak solution, we should use an approximation scheme that is monotone~\cite{CrandallLions_HJApprox}.
We also observe that this equation can be formulated as an optimal control problem~\cite{Bardi_HJB}.  In particular, we can rewrite it as a Hamilton-Jacobi-Bellman equation
\[ \max\{u_x - \sqrt{u}, -u_x - \sqrt{u}\} = 0. \]

To sweep from the left, we solve the ODE
\[
\begin{cases}
u_x - \sqrt{u} = 0, & x>0\\
u = 1, & x = 0,
\end{cases}
\]
which gives us the left solution branch
\[ u_-(x) = \frac{1}{4}x^2+x+1. \]
Similarly, we can sweep from the right to obtain the right solution branch
\[ u_+(x) = \frac{1}{4}x^2-\frac{3}{2}x+\frac{9}{4}. \]
Once these solution branches have been generated, we match them using the selection principle
\bq\label{eq:eikonalSol} u(x) = \min\left\{u_-(x),u_+(x)\right\} = \min\left\{\frac{1}{4}x^2+x+1,\frac{1}{4}x^2-\frac{3}{2}x+\frac{9}{4}\right\}, \eq
which is a consequence of the optimal control formulation.

We can make a connection with a simple one-dimensional scalar conservation law by differentiating the Hamilton-Jacobi equation with respect to $x$ and defining the variable $v = u_x$.  This gives us the one-dimensional Burger's equation
\[ \left(v^2\right)_x = v. \]
By referring to the original Hamilton-Jacobi equation, together with the definition of $v$, we can also obtain the boundary conditions
\[ v(0) = 1, \quad v(1) = -1. \]
Solving the conservation law from the left and right boundaries respectively, we obtain the solution branches
\[ v^-(x) = \frac{1}{2}x+1, \quad v^+(x) = \frac{1}{2}x-\frac{3}{2}. \]

Even in the simple setting of this scalar one-dimensional conservation law, there is no direct generalisation of the selection principle that we used for the Hamilton-Jacobi equation.  
In addition, systems of conservation laws and multi-dimensional problems do not share the same link with Hamilton-Jacobi equations.  Nevertheless, the efficacy of fast sweeping methods motivates us to consider alternative selection principles that will allow us to use a similar approach for solving systems of conservation laws.

We will provide details about the proposed selection principle beginning in~\S\ref{sec:shock1d}.  For now, we simply state that the Rankine-Hugoniot condition that must hold at a shock in a stationary solution of Burger's equation is
\[ (v^{-}(x))^2 = (v^+(x))^2. \]
In the example we consider here, this equation has the solution $x = \frac{1}{2}$ and the entropy solution of the conservation law is
\[ v(x) = \begin{cases}
\frac{1}{2}x+1, &0 < x < \frac{1}{2}\\
\frac{1}{2}x-\frac{3}{2}, & \frac{1}{2}<x<1.
\end{cases} \]

\subsection{Shock Conditions: One Dimension}\label{sec:shock1d}
For systems of conservation laws, the weak solutions need not be continuous.  In general, we can expect the different solution branches to meet in a shock.  Whatever selection principle is used must satisfy the appropriate conservation conditions, which are equivalent to the Rankine-Hugoniot shock conditions.  It is well known that the Rankine-Hugoniot conditions alone may not be sufficient for describing a valid shock, and an additional entropy condition must also be verified.

We begin by considering the selection principle for steady state solutions of the  one-dimensional system
\bq\label{system1d}
U_t + f(U)_x = a(U,x).
\eq

In one-dimension, the Rankine-Hugoniot conditions give a condition for the shock speed $s$:
\[ s[[U]] = [[f(U)]]. \]
Here, we use
\[ [[v]] = v_+ - v_- \]
to denote the jump in a quantity across a shock.

Since we are concerned with steady state solutions, we are only interested in computing stationary shocks; that is, the shock speed should vanish.  The selection principle that can be used to determine a valid shock location thus becomes
\bq\label{eq:RH_1D}
[[f(U)]] = 0.
\eq

In addition to this condition on stationary shocks, we also require that characteristics are entering rather than emanating from the shock; this is the entropy condition.  To describe the Lax entropy condition, we first need to recall that for a hyperbolic system, the Jacobian of the flux $\nabla f(U)$ has real eigenvalues,
\[ \lambda_1 < \lambda_2 < \ldots < \lambda_n.\]
A stationary shock in the $k^{th}$ characteristic field is required to satisfy the Lax shock conditions~\cite{Lax}:
\bq\label{eq:entropy}
\begin{split}
\lambda_k(U_+) &< 0 < \lambda_k(U_-)\\
\lambda_{k-1}(U_-) &< 0 < \lambda_{k+1}(U_+).
\end{split}
\eq

\subsection{Shock Conditions: Two Dimensions}\label{sec:shock2d}
Next we consider steady state solutions of the two-dimensional system
\bq\label{eq:system2d}
U_t + f(U)_x + g(U)_y = a(U,x).
\eq

In two dimensions, a shock occurs along a curve instead of at a point.  Now the one-dimensional conditions will be applied in the direction $n$ normal to the curve.  Thus the Rankine-Hugoniot conditions for a stationary shock are
\bq\label{eq:RH_2D}
n \cdot\left([[f]],[[g]]\right) = 0.
\eq

In the case of two-dimensional scalar equations, we again require that characteristics are directed in towards the shock.  If we suppose that the solutions on either side of the shock are given by $U_-$ and $U_+$, and that the normal vector $n$ is directed towards the positive side of the shock (where $U=U_+$), then the entropy condition~\cite{ZhengBook} becomes
\bq\label{eq:entropy2D} 
n \cdot \left(f'(U_+),g'(U_+)\right) < 0 < n \cdot \left(f'(U_-),g'(U_-)\right).
\eq

%For multi-dimensional systems of conservation laws, the characteristic structure can become more complex and the well-posedness theory is underdeveloped.  We are not aware of a similar characterisation of the entropy solution in this most general setting.

\section{One-Dimensional Problems}\label{sec:1d}
We begin by describing our sweeping approach for obtaining steady state solutions of the one-dimensional system of conservation laws,
\[ U_t + f(U)_x = a(U,x), \quad x_L < x < x_R, \]
together with appropriate boundary conditions.

After describing the assumptions we make on the data, we will use several examples to describe our sweeping approach.  A more general discussion of one-dimensional problems will be given in~\S\ref{sec:structure}.

\subsection{Assumptions}\label{sec:assumptions}
In the simplest setting, the boundary conditions
\[ B(U,x) = 0, \quad x\in\partial\Omega \]
can be inverted so that all components of the solution vector $U$ are explicitly prescribed on the boundary.
However, in many situations this will lead to an overdetermined (and likely ill-posed) problem.

To determine how many boundary conditions are really needed, we need to look at the orientation of the characteristic fields at the boundary points.  Recall that the signs of the eigenvalues of $\nabla f$,
\[ \lambda_1 < \lambda_2 < \ldots < \lambda_n, \]
determine whether information is traveling from left to right or from right to left along the corresponding characteristic.  Thus we expect that on the left boundary, the number of positive eigenvalues should correspond to the number of components of $U$ that are being propagated into the domain, which should in turn correspond to the number of boundary conditions given on the left side of the domain.  Similarly, at the right boundary point, we will assume that the number of boundary conditions is equal to the number of negative eigenvalues.

To make this more concrete, we suppose that on the left side of the domain the first $I$ eigenvalues are negative,
\[ \lambda_1< \lambda_2 < \ldots < \lambda_I < 0 < \lambda_{I+1} < \ldots < \lambda_n, \quad x = x_L. \]
Then we assume that the boundary condition
\[ B_L(U) = 0, \quad x = x_L \]
determines $n-I$ components of $U$ at the left boundary.  That is, 
\[ U = U_L^{\alpha_1,\ldots,\alpha_I}, \quad x = x_L \]
has $I$ degrees of freedom in the form of the unknown parameters $\alpha_1,\ldots,\alpha_I$.  Similarly, if the first $J$ eigenvalues are negative at the right boundary,
\[ \lambda_1< \lambda_2 < \ldots < \lambda_J < 0 < \lambda_{J+1} < \ldots < \lambda_n, \quad x = x_R \]
then the boundary condition
\[ B_R(U) = \left(\begin{tabular}{c}$B_R^1(U)$\\ \vdots \\$B_R^J(U)$\end{tabular}\right) = 0, \quad x = x_R \]
should provide $J$ conditions at the right boundary.

Furthermore, we will primarily focus our attention on problems with steady state solutions that contain a single shock in the interior of the domain.  However, the approach we describe can also be generalised to problems with multiple stationary shocks using the reasoning in~\S\ref{sec:structure}.

\subsection{Generating a Solution Branch}\label{sec:branch1d}
In the overview of sweeping methods given so far~(\S\ref{sec:sweeping}), we suggested a technique of sweeping solutions in from the boundaries and using the Rankine-Hugoniot condition to select the appropriate solution branch.  This is a good strategy in multi-dimensions and for one-dimensional scalar problems. For one-dimensional systems, however, it is often preferable to modify this technique by just sweeping in one dimension.

In this variant of the method, sweeping is done starting at the side of the domain that has the most boundary conditions.  Naturally, a similar procedure could be used to solve from right to left instead.  In the interior of the domain, the Rankine-Hugoniot conditions are used to switch between smooth solution branches.  The given boundary data at the far side of the domain is used to determine the correct shock location.

Whichever form of sweeping we use, we require a procedure for computing a smooth solution branch starting either at a boundary point or a shock.  We describe the procedure for sweeping from left to right; sweeping from right to left is similar.  If we are given full boundary conditions $U_L$ at the left boundary point $x_L$, we can propagate these into the domain by solving the problem
\bq\label{eq:ODEs}
\begin{cases}  
f(U)_x = a(U,x), & x_L < x \leq x_R\\
U = U_L, & x = x_L.
\end{cases}
\eq
As long as the flux $f(U)$ is locally invertible, this is equivalent to solving the system of ODEs
\[\begin{cases}  
V_x = a(f^{-1}(V),x), & x_L < x \leq x_R\\
V = f(U_L), & x = x_L
\end{cases}\]
with $U = f^{-1}(V)$.

These ODEs can be solved using any suitable method.  In the computations that follow, we simply use forward Euler.  The flux function is easily inverted using a Newton step.  Recall that we are sweeping from left to right to generate a continuous solution branch.  Thus if we want to solve
\[ f(U_j) = V_j \]
for $U_j$ at the grid points $x_j$, we can initialise the Newton solver with the neighbouring value $U_{j-1}$.  

\subsubsection{Isentropic flow through a duct (left branch)}\label{sec:isentropicLeft}
To illlustrate the approach we have just described, we consider the equations for isentropic flow through a duct.
\bq\label{eq:isentropic}
\left(\begin{tabular}{c}$\rho$\\$m$\end{tabular}\right)_t + \left(\begin{tabular}{c}$m$\\$\frac{m^2}{\rho}+\kappa \rho^\gamma$\end{tabular}\right)_x = \left(\begin{tabular}{c}$-\frac{A'(x)}{A(x)}m$\\$-\frac{A'(x)}{A(x)}\frac{m^2}{\rho}$\end{tabular}\right).
\eq
The eigenvalues of this system are 
\[ u \pm c = \frac{m}{\rho} \pm \sqrt{\kappa\gamma\rho^{\gamma-1}}.\]

We choose the constants $\gamma=1.4$ and $\kappa = 1$ and 
we let the cross-sectional area of the duct be given by
\[ A(x) = -\frac{2}{5}\cos(\pi x)+\frac{6}{5}. \]

We further consider the situation of left to right flow that is supersonic at the left boundary $x=0$ and subsonic at the right boundary $x=1$.  This means that the eigenvalues will satisfy
\begin{align}
 0<\lambda_1<\lambda_2, \quad x=0,\label{eq:eigLeft} \\
\lambda_1<0<\lambda_2, \quad x = 1.\label{eq:eigRight}
\end{align}

  The given boundary conditions are
\[ m(0) = 2, \quad \rho(0) = 1, \quad \rho(1) = 2. \]

As described in this section, we can compute the smooth solution $(\rho_-,m_-)$ of the ODEs~\eqref{eq:ODEs} with initial conditions given by $m(0)$ and $\rho(0)$; this gives us a left solution branch in the region $x>0$ (Figure~\ref{fig:shock}).  Note, however, that the density does not satisfy the given boundary condition 
\[ \rho = 2, \quad x = 1 \]
on the right side of the domain.  It will be necessary to introduce a shock into the solution in order to produce a solution that satisfies all boundary conditions.

\subsection{Enforcing Shock Conditions}\label{sec:RH_1d}
As we have just seen, if we are given a boundary condition
\bq\label{eq:rightBC} B_R(U) = 0, \quad x = x_R \eq
on the right, we cannot expect to compute the correct solution by sweeping once from the left.  That is, in general we will find that
\[ B_R(U_-) \neq 0, \quad x = x_R.\]
  Instead, a shock will need to be introduced in order to ensure that all boundary conditions are satisfied.

Let us suppose first of all that we have a candidate shock location $x_S$.  Given the entropy conditions in~\eqref{eq:entropy}, we know that the shock must occur in the characteristic field corresponding to the smallest positive eigenvalue.

As discussed earlier, one option is to generate the left solution branch $U_-$ and the right solution branch $U_+$.  Then the unknown shock location $x_S$ is chosen as the point where the Rankine-Hugoniot condition is satisfied,
\[ f(U_-) = f(U_+). \]
This is a simple approach for a 1D scalar problem.  However, as per the discussion in~\S\ref{sec:assumptions}, we may not have full boundary values prescribed on both sides of the domain.  In this case, we cannot directly compute left and right solution branches.  Instead, the left solution branch will depend on $I$ unknown parameters, while the right solution branch will depend on $n-J$ unknown parameters.  The shock location is an additional unknown, which results in a need to determine $I+n-J+1$ unknowns.

A more efficient approach is to sweep in one direction only, starting from the side that has the most boundary conditions prescribed.  Throughout this paper, we will describe a left-to-right sweeping procedure, but the right-to-left procedure is analogous.  This will reduce the number of unknowns to $I+1$: the $I$ free parameters in the solution vector at $x_L$ and the location of the shock $x_S$.

We start by supposing that we have the full solution vector at the left endpoint $x_L$.  The more general setting will be considered in~\S\ref{sec:sonic}-\ref{sec:structure}.

The solution values on the left side of the shock are given by the left branch we have generated: $U_{-}(x_S)$.  We also need to determine the values $\Phi\left(U_-(x_S)\right)$ on the right side of the shock.  This is done by looking for entropy-satisfying solutions of the Rankine-Hugoniot conditions:
\bq\label{eq:jump} f\left(\Phi\left(U_-(x_S)\right)\right) = f(U_-(x_S)).\eq
We make a couple observations about the jump operator $\Phi$:
\begin{enumerate}
\item If the system~\eqref{eq:jump} has no solutions that satisfy the entropy conditions~\eqref{eq:entropy}, then $x_S$ is not an allowed shock location.
\item The entropy conditions~\eqref{eq:entropy} require the presence of a shock, so that $\Phi\left(U_-(x_S)\right)=U_-(x_S)$ is not an admissible solution of~\eqref{eq:jump}.
\end{enumerate}

Once this has been done, we can continue to propagate the solution from left to right, starting at the shock, by solving the ODEs
\bq\label{eq:ODERight}
\begin{cases}  
f(U)_x = a(U,x), & x_S < x \leq x_R\\
U = \Phi\left(U_-(x_S)\right), & x = x_S.
\end{cases}
\eq
Let us denote by $U(x; x_S)$ the solution generated if there is a shock at the point $x_S$.

For arbitrary shock locations $x_S$, we cannot expect $U(x; x_S)$ to satisfy the given boundary condition~\eqref{eq:rightBC}; see Figure~\ref{fig:shock}.
In order to determine the correct shock location, we need to make use of this boundary condition.
Thus the problem becomes to find the unknown $x_S$ such that the resulting solution of~\eqref{eq:ODERight} satisfies the equation
\bq\label{eq:rigthEqn} B_R(U(x; x_S)) = 0, \quad x = x_R. \eq
This can be done, for example, using a bisection method since the solution $U(x; x_S)$ of~\eqref{eq:ODERight} depends continuously on the value of $\Phi\left(U_-(x_S)\right)$, which in turn depends continuously on the single parameter $x_S$ through~\eqref{eq:jump}.   

If the bisection is only used to provide a starting point for a faster algorithm, such as Newton's method, the overall computational complexity of this procedure would be $\bO(N)$.
%The bisection scheme itself requires $\bO(\log N)$ iterations.  Each iteration necessitates the solution of an ODE~\eqref{eq:ODERight}, which requires $\bO(N)$ time. Thus the overall computational complexity will be $\bO(N\log N)$. 

We also remark that in our computations, we solve for the shock location to within the nearest grid point.  However, if even more accurate shock tracking is desired, we could use a smaller step size in the vicinity of the shock to refine the approximation of the shock location.

\begin{figure}[htdp]
	\centering
	\subfigure[]{\includegraphics[width=.48\textwidth]{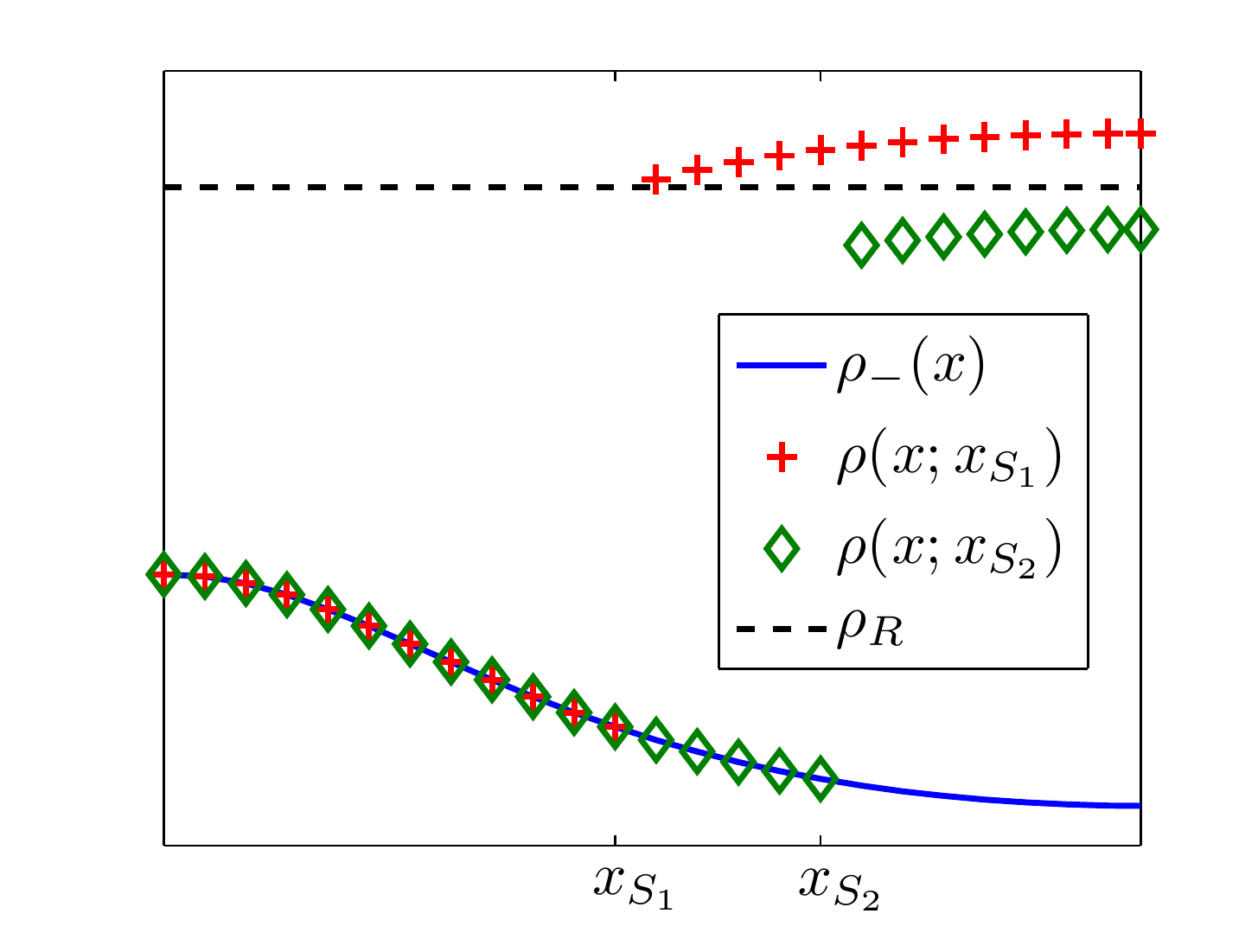}\label{fig:shock}}
	\subfigure[]{\includegraphics[width=.48\textwidth]{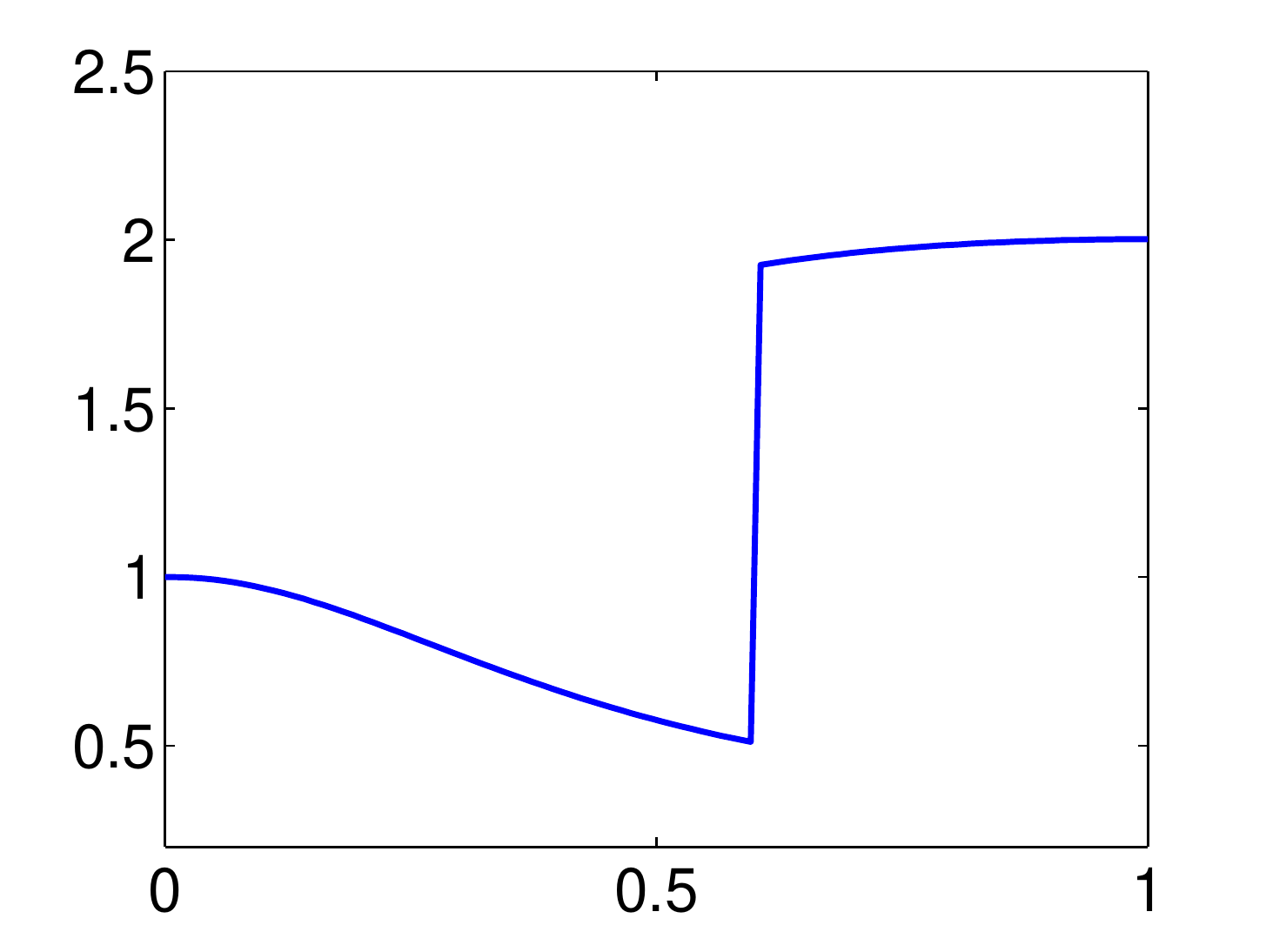}\label{fig:isentropic_rho}}
  	\caption{
  	\subref{fig:shock}~Left solution branch $\rho_-(x)$, solutions $\rho(x; x_{S_1}),\rho(x; x_{S_2})$ obtained by imposing a shock at $x_{S_1},x_{S_2}$, and the given right boundary value.  \subref{fig:isentropic_rho}~Computed density for the isentropic equations in~\S\ref{sec:isentropicLeft},~\ref{sec:isentropic1}.} 
  	\label{fig:isentropic}
\end{figure}

\subsubsection{Isentropic flow through a duct (unique solution)}\label{sec:isentropic1}
We return now to the isentropic equations~\eqref{eq:isentropic}.

As described in~\S\ref{sec:branch1d}, we can generate the smooth left solution branch.  This step only needs to performed once.

The Rankine-Hugoniot conditions at a stationary shock are
\[ m_- = m_+, \quad \frac{m_-^2}{\rho_-} + \kappa\rho_-^\gamma = \frac{m_+^2}{\rho_+} + \kappa\rho_+^\gamma. \]
Clearly, the momentum $m$ is continuous across the shock.  The second equation has 
two solutions.  One of these is $\rho_- = \rho_+$, which we discard since we are looking for a shock.  The second, desired root is easily obtained using Newton's method.

We use a bisection method to choose a shock location that enforces the condition $\rho(1)=2$.  The computed density is plotted in Figure~\ref{fig:isentropic_rho}.  We also plot the eigenvalues of $\nabla f$, which make clear that the solution is an entropy stable 1-shock.  Finally, we present computation times in Table~\ref{table:isentropic} to validate our claims about the efficiency of our approach.

\begin{table}[htdp]
\caption{Computation time using $N$ grid points for the isentropic equations in~\S\ref{sec:isentropic1}.}
\begin{center}
\begin{tabular}{c||cccccc}
N  & 64 & 128 & 256 & 512 & 1024 & 2048 \\
\hline
CPU Time (s) & 0.8 & 1.2 & 2.4 & 5.1 & 11.4 & 23.5\\
\end{tabular}
\end{center}
\label{table:isentropic}
\end{table}

\subsection{Problems with Multiple Steady States}\label{sec:multiple}
As discussed in~\cite{EmbidMajda_Mult}, conservation laws need not have unique steady state solutions.  Instead, the steady states can depend on the initial values.  We should note that it is also possible for a problem to have a valid entropy-stable steady state solution that is not time-stable and thus cannot be generated through time evolution of a conservation law.

We want our methods to generate all valid shock solutions.  This simply means that when we are choosing the shock location required to satisfy the given right boundary conditions, we should be aware of the possibility of multiple solutions.  

\subsubsection{Isentropic flow through a duct (multiple solutions)}\label{sec:isentropic2}
To illustrate this, we return to the problem of isentropic flow through a duct, which was introduced in~\S\ref{sec:isentropicLeft}.  We consider exactly the same problem, but with a new duct geometry
\[ A(x) = 1.2-0.2\cos(4\pi x). \]
This has the effect of introducing oscillation into the source term, which allows for multiple valid stationary shock locations.

We repeat the procedure of the preceding section.  As before, we generate the left branch from the data.  Next we split the domain into four sub-regions where the source term does not change sign.  We search for a shock in each of these regions, using the endpoints to initialise our bisection method.

This allows us to compute four distinct solutions; see Figure~\ref{fig:multiple}.  We also note that it appears that only the first and third of these solutions are time-stable.

\begin{figure}[htdp]
	\centering
			\subfigure[]{\includegraphics[width=.44\textwidth]{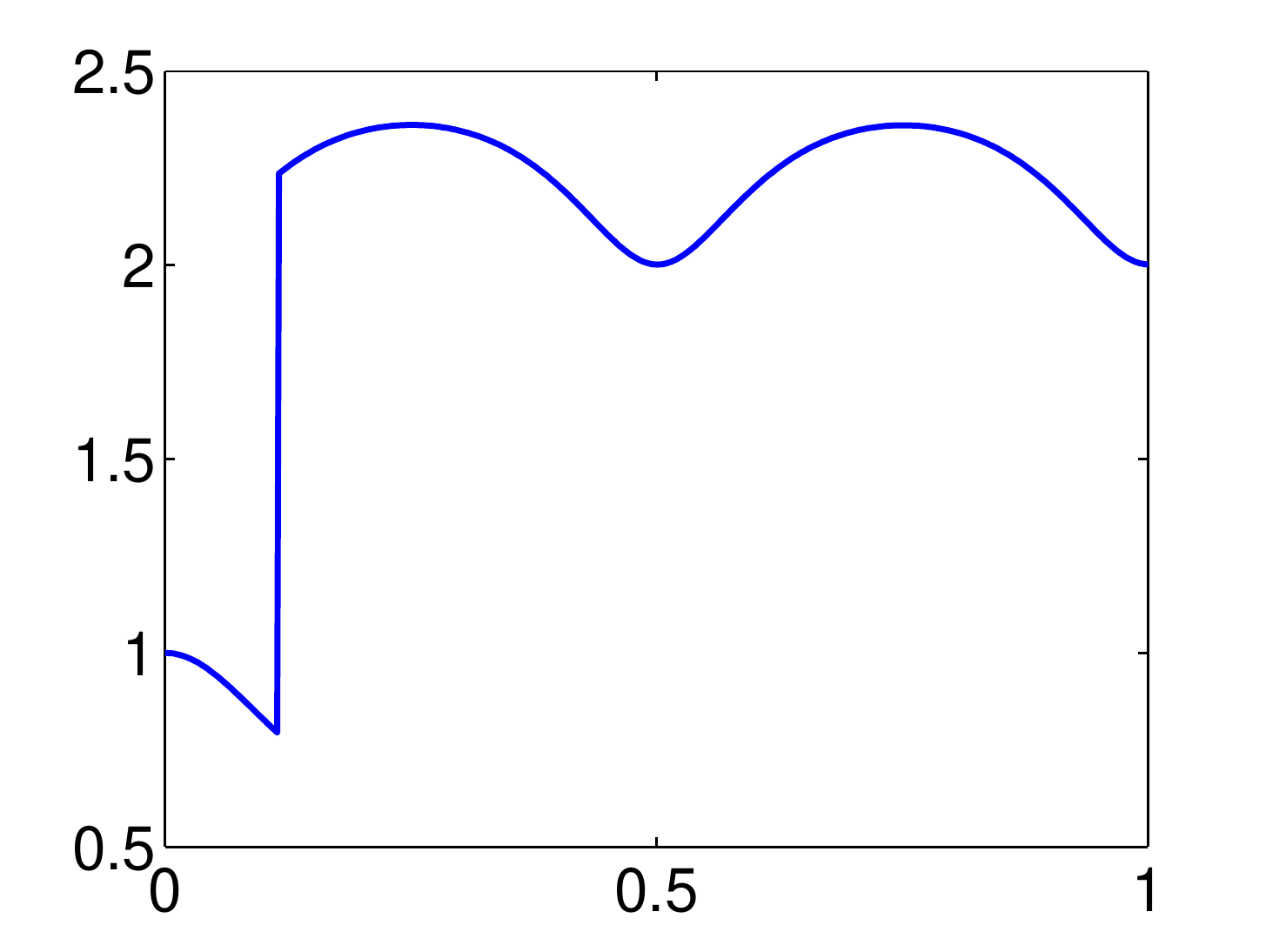}\label{fig:rho_mult1}}
%       \subfigure[]{\includegraphics[width=.44\textwidth]{lam_mult1}\label{fig:lam_mult1}}
       \subfigure[]{\includegraphics[width=.44\textwidth]{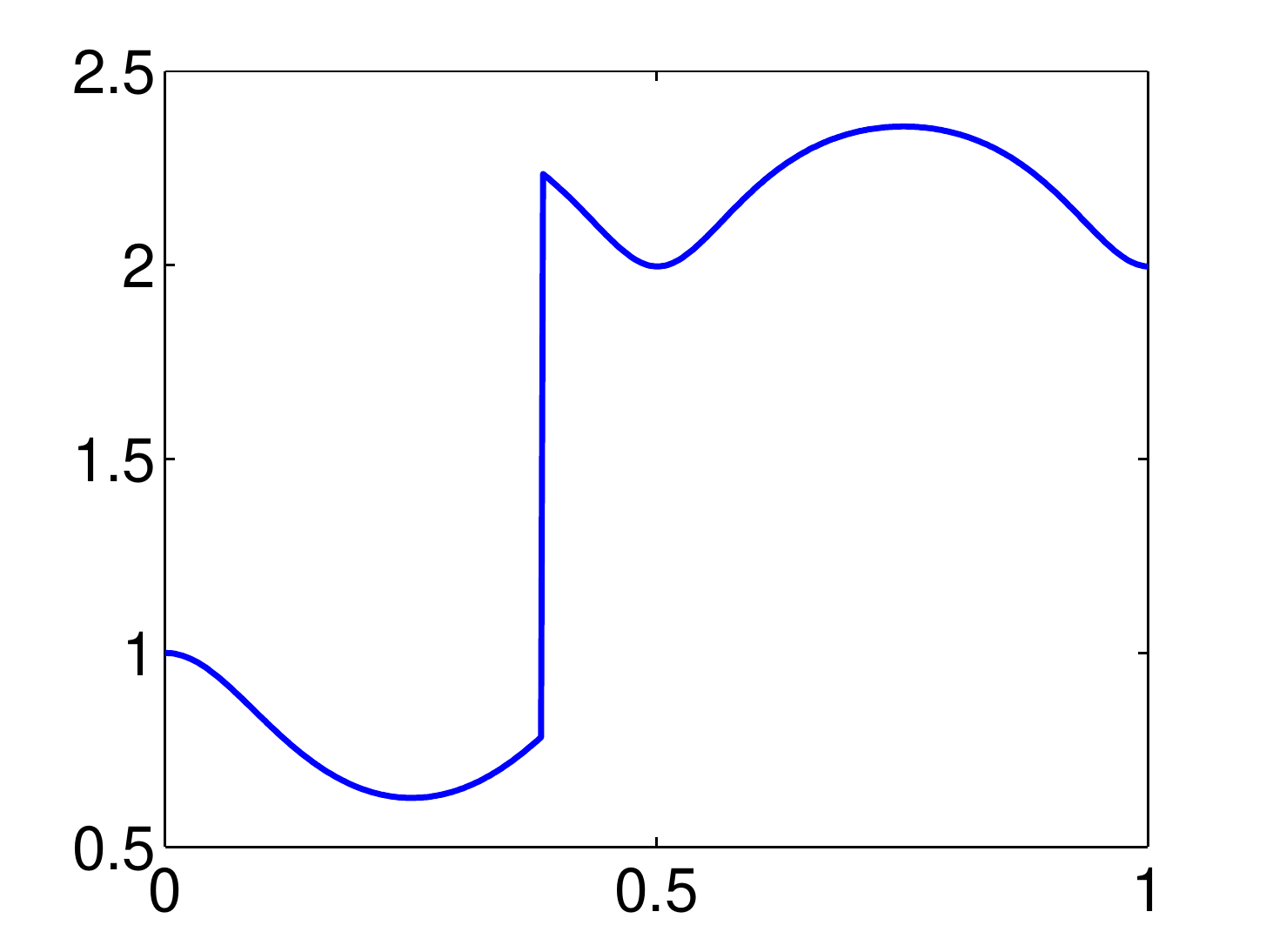}\label{fig:rho_mult2}}
%       \subfigure[]{\includegraphics[width=.44\textwidth]{lam_mult2}\label{fig:lam_mult2}}
       \subfigure[]{\includegraphics[width=.44\textwidth]{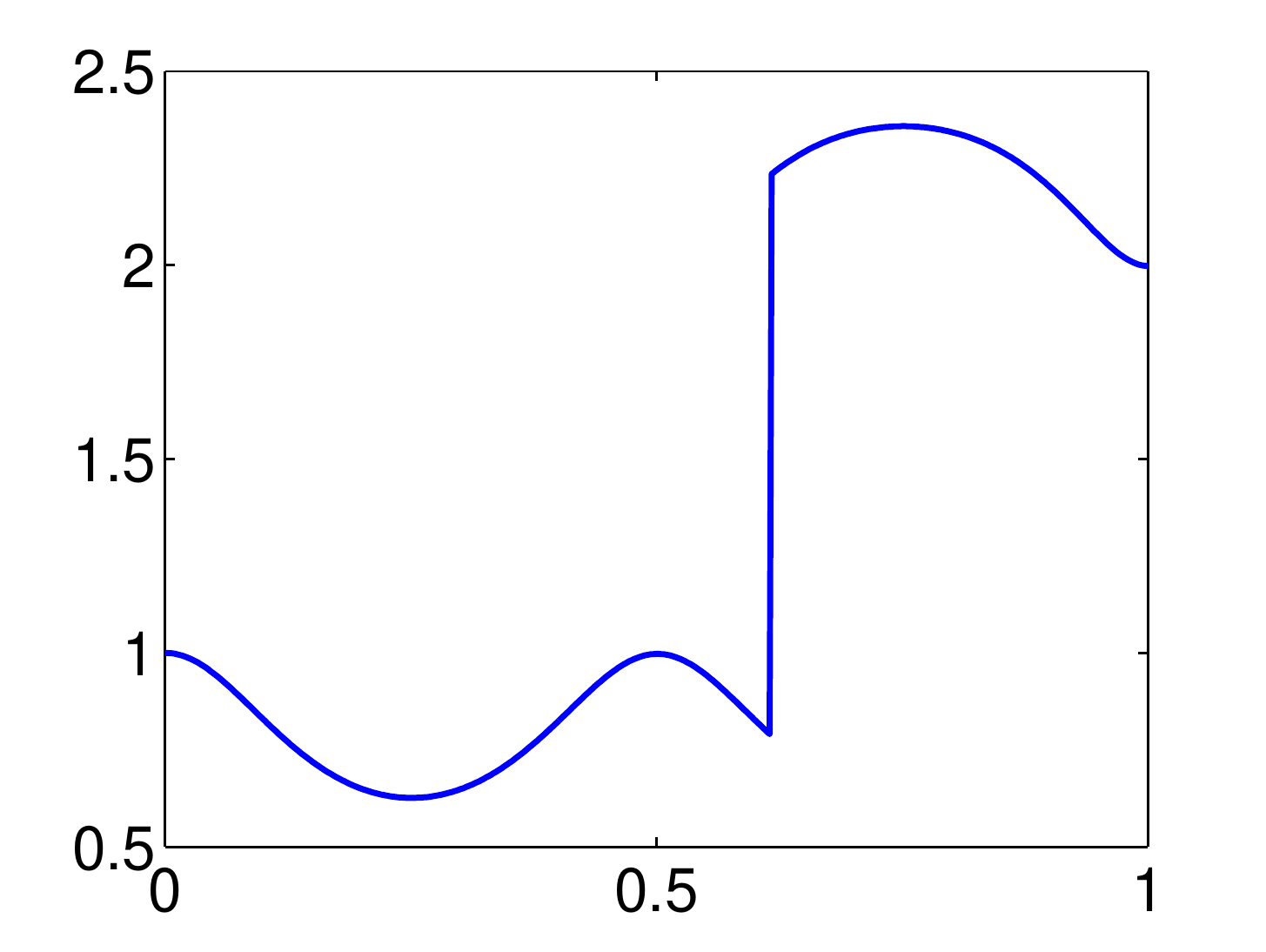}\label{fig:rho_mult3}}
%       \subfigure[]{\includegraphics[width=.44\textwidth]{lam_mult3}\label{fig:lam_mult3}}
       \subfigure[]{\includegraphics[width=.44\textwidth]{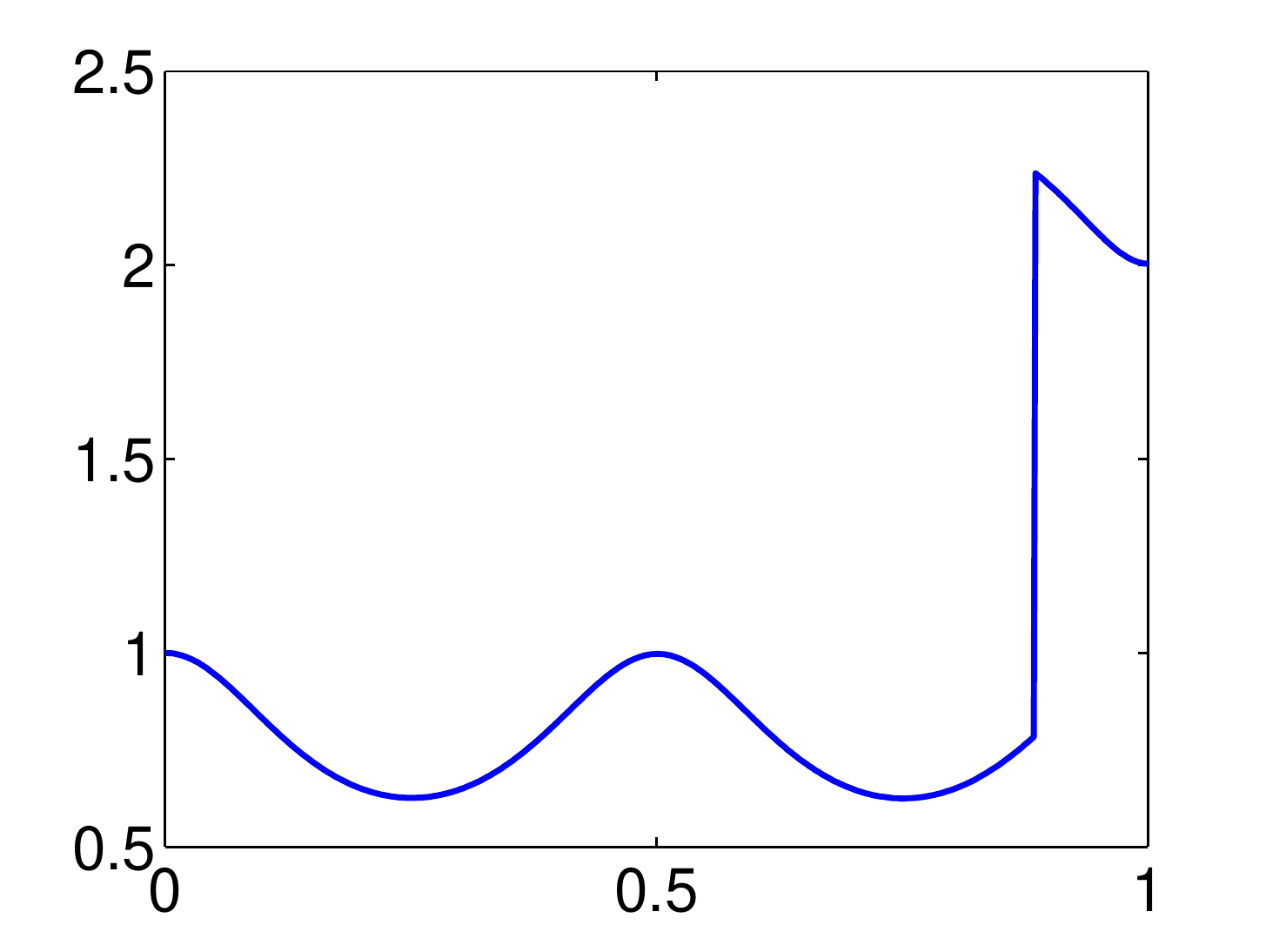}\label{fig:rho_mult4}}
%       \subfigure[]{\includegraphics[width=.44\textwidth]{lam_mult4}\label{fig:lam_mult4}}
  	\caption{Density for four different stationary solutions of the isentropic equations in~\S\ref{sec:isentropic2}.}
  	\label{fig:multiple}  	
\end{figure}

\subsection{Sonic Points}\label{sec:sonic}
We recall that in the sweeping procedure we described in~\S\ref{sec:branch1d}, we assumed that $\nabla f(U)$ is invertible.  This will be reasonable as long as none of the eigenvalues of $\nabla f$ vanish.  However, it is also possible for one of these eigenvalues to change sign continuously, passing through zero in the process.  As we approach this turning point $x_T$, where an eigenvalue changes sign, the system of ODEs~\eqref{eq:ODEs} becomes very stiff, and conventional ODE methods will not allow us to solve the system up to (or beyond) $x_T$. We will describe an alternative approach.

Let us consider the problem of sweeping the conservation law
\bq\label{eq:1dsteady} f(U)_x = a(U,x) \eq
from left to right, where the $i^{th}$ characteristic encounters a sonic (turning) point at some point $x_T$.

As long as we are to the left of $x_T$, we can solve this ODE using the procedure described in the previous sections.

We also want to continue to evolve the ODEs through the turning point.  To gain insight into whether or not this is possible, it is helpful to look at the linearised system
\[ \Lambda(P^{-1}U)_x \approx P^{-1}a(U,x) \]
where
\[ P^{-1}\nabla f(U_T)P = \Lambda \]
and $\Lambda$ is a diagonal matrix containing the eigenvalues of $\nabla f$.
Since the $i^{th}$ eigenvalue vanishes, it will also be necessary for the $i^{th}$ component of $P^{-1}a$ to vanish at the sonic point.

Thus we cannot hope to generate a continuous solution through the turning point unless the compatibility condition
\bq\label{eq:compatibility}
P^{-1}a(U_T,x_T) = 0
\eq
is satisfied.

Thus given full left boundary conditions (that is, all components of $U_-(x_L)$ are given as data), we cannot in general expect the resulting smooth solution branch $U_-(x)$ to satisfy this compatibility condition.  However, if we are missing the boundary condition corresponding to the $i^{th}$ characteristic field, we can use this extra degree of freedom to choose boundary conditions that will allow us to satisfy the compatibility condition at the turning point. 
That is, suppose the boundary condition
\[ B_L(U) = 0, \quad x = x_L  \]
has a one-parameter family of solutions $U_L^\alpha$.  For a given value of the parameter $\alpha$, we can solve the system of ODEs~\eqref{eq:ODEs} in the domain $x < x_L < x_T^\alpha$ to obtain a left solution branch $U_-^\alpha(x)$.  This unknown parameter is then determined by the compatibility condition~\eqref{eq:compatibility}
\[  P^{-1}a\left(U_T^\alpha,x_T^\alpha\right) = 0. \]
Formally, we have one unknown (a boundary condition), which is determined by one equation (the compatibility condition).  Thus in general, we expect that an eigenvalue could transition from negative to positive through a sonic point.

At the discrete level, we can use a conventional ODE solver to solve the system of ODES~\eqref{eq:ODEs} from $x_L$ up to a grid point $x_j<x_T$ that is near the sonic point.  Since we are approximating a smooth solution, a forward Euler formula will be valid even at the (unknown) turning point:
\[ f(U_T) - f(U_j) \approx \left(x_T-x_j\right) a(U_j,x_j). \]
We also require the $i^{th}$ eigenvalue to vanish at the turning point:
\[ \lambda_i(U_T) = 0. \]
We can use these equations to solve not only for the solution $U_T$ at the turning point, but also for the location $x_T$ of the turning point.   If we are looking at a system of $n$ conservation laws, this leads to a system of $n+1$ equations for $n+1$ unknowns, which are the turning point $x_T$ and the $n$ components of the solution $U_T$.  This system can be solved using Newton's method, with the nearby grid location $x_j$ and solution values $U_j$ providing a good initial guess.

Once we have solved the system from $x_L$ to the turning point $x_T$, we can use a backward Euler formula to obtain the solution at the next grid point:
\[f(U_{j+1}) - f(U_T) = (x_{j+1}-x_T)a(U_{j+1},x_{j+1}).\]
We can again invert these with Newton's method, but we do have to be careful to obtain the correct solution since there will be an issue of non-uniqueness near the sonic point.  We extrapolate to obtain the initial guess 
\[U_{j+1} \approx \frac{x_{j+1}-x_j}{x_T-x_j}U_T +\left(1-\frac{x_{j+1}-x_j}{x_T-x_j}\right) U_j.\]

Once this is done, we can continue to sweep this solution branch towards the right using any suitable ODE solver.

\subsubsection{Nozzle problem}\label{sec:nozzle}
To illustrate the issues surrounding sonic points and missing boundary conditions, we consider the nozzle problem.
\bq\label{eq:nozzle} 
\left(\begin{tabular}{c}$\rho A$\\ $\rho u A$\\ $EA$\end{tabular}\right)_t + \left(\begin{tabular}{c}$\rho u A$\\ $(\rho u^2 + p) A$\\ $uA(E+p)$\end{tabular}\right)_x = \left(\begin{tabular}{c}0\\$pA'(x)$\\ 0\end{tabular}\right), 
\eq
which we want to solve to steady state on the domain $x\in[0,3]$.

Here the pressure is
\[ p = (\gamma-1)\left(E-\frac{1}{2}\rho u^2\right) = \rho R T \]
and the sound speed is
\[ c = \sqrt{\gamma p/\rho}. \]
The eigenvalues of the Jacobian are $\lambda_1 = u-c$, $\lambda_2 = u$, and $\lambda_3 = u+c$.

Following~\cite{Chen_LFSweeping}, we take the cross-sectional area to be
\[ A(x) = 1+2.2(x-1.5)^2, \]
the gas constant $\gamma = 1.4$, and $R = 8.3144$.

We consider the boundary conditions 
\[ p_L = 1, \quad p_R = 0.6784, \quad T_L  = 300. \]
Given these boundary conditions, we expect that $\lambda_1<0<\lambda_2<\lambda_3$ on both sides of the domain.

Since the eigenvalues have the same signs on both sides of the domain, one possibility to consider is that the missing left boundary condition should be chosen so that the resulting (smooth) left branch satisfies the given right boundary condition.  We set this possibility aside since we are interested in producing a solution with a shock, and in illustrating the effects of sonic points.

If we consider the physically reasonable setting where the (steady) flow is from left to right ($u>0$), we can make several observations about the structure of a solution with a shock.
\begin{enumerate}
\item A stationary shock can only occur in the first characteristic field since $\lambda_1 = u-c$ is the only eigenvalue that is permitted to become negative.
\item The entropy conditions~\eqref{eq:entropy} require that $\lambda_1>0$ to the immediate left of the shock.
\item The given boundary conditions assume that $\lambda_1<0$ at $x = x_L$, the far left of the domain.
\item We conclude that the first eigenvalue $\lambda_1$ must change sign through a sonic point before a shock can occur.
\end{enumerate}

Now we want to choose the unknown boundary value in order to ensure that the compatibility condition is satisfied at the turning point, which for this problem means
\[ -\frac{(\gamma-1)u+c}{4\gamma(\gamma-1)}\rho A'(x) = 0. \]
Since we are interested in left to right flow, this is equivalent to $A'(x)=0$.

Let use denote by $U_L^\alpha$ the left boundary values, with one free parameter $\alpha$, and by $x_T^\alpha$ the location of the resulting sonic point.  We are looking for the value of $\alpha$ that ensures that
\[  A'(x_T^\alpha) = 0. \]
To solve this, we define
\[ x_*^\alpha = 
\begin{cases}
x_T^\alpha & \text{if there is a sonic point $x_T$ in the domain}\\
x_R & \text{otherwise}.
\end{cases} \]
Then we use a bisection method to solve $A'(x_*^\alpha) = 0$ for $\alpha$.

We should note that in this problem, $x_*^\alpha$ is not continuous as a function of $\alpha$.  However, the bisection scheme will still converge to a value where $A'(x_*^\alpha)$ changes sign, which is the sonic point.

Once this is done, we can generate a left solution branch $U_-(x)$ in the entire domain.

From this point, we solve for the unknown shock location as in the previous examples.  With the use of the bisection methods, the entire solution procedure requires $\bO(N\log N)$ time; this is supported by the computation times in Table~\ref{table:nozzle}.  The computed solution, as well as a reference solution obtained by evolving the time-dependent problem to steady state, are presented in Figure~\ref{fig:nozzle}.  Of particular note is the sharp shock that the sweeping method produces.

\begin{figure}[htdp]
	\centering
			\subfigure[]{\includegraphics[width=.48\textwidth]{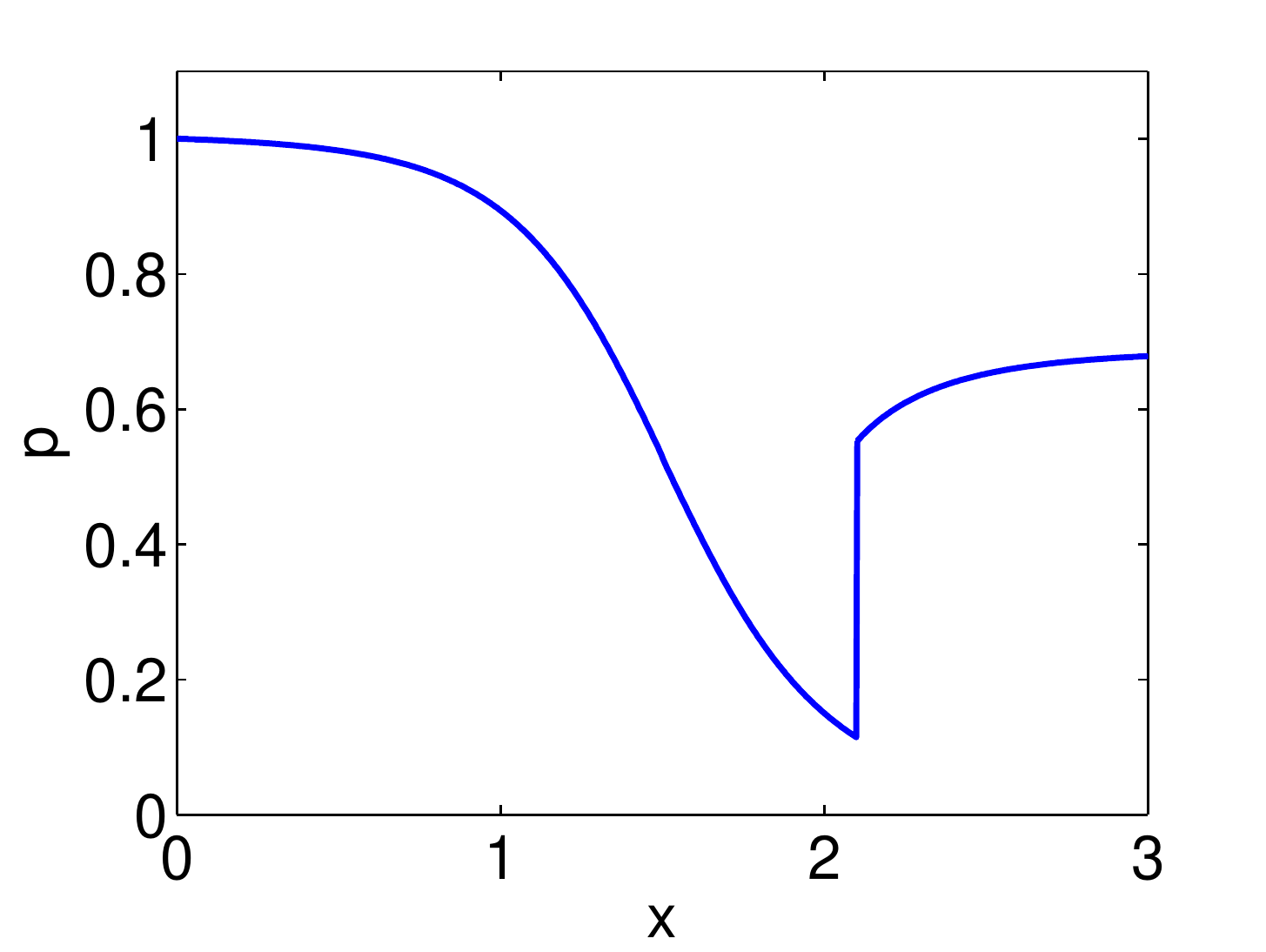}\label{fig:p_nozzle_1000}}
       \subfigure[]{\includegraphics[width=.48\textwidth]{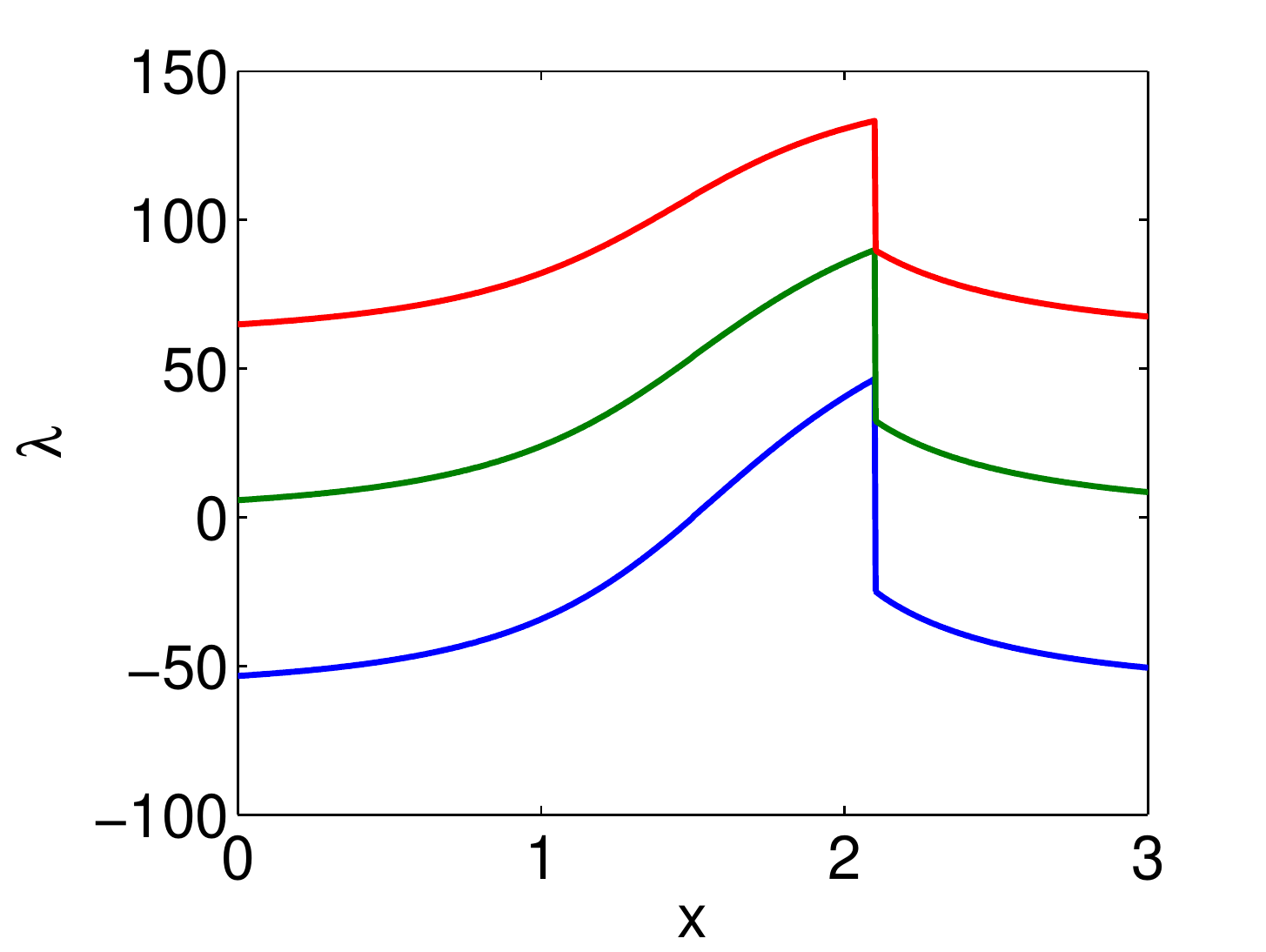}\label{fig:l_nozzle_1000}}
       \subfigure[]{\includegraphics[width=.48\textwidth]{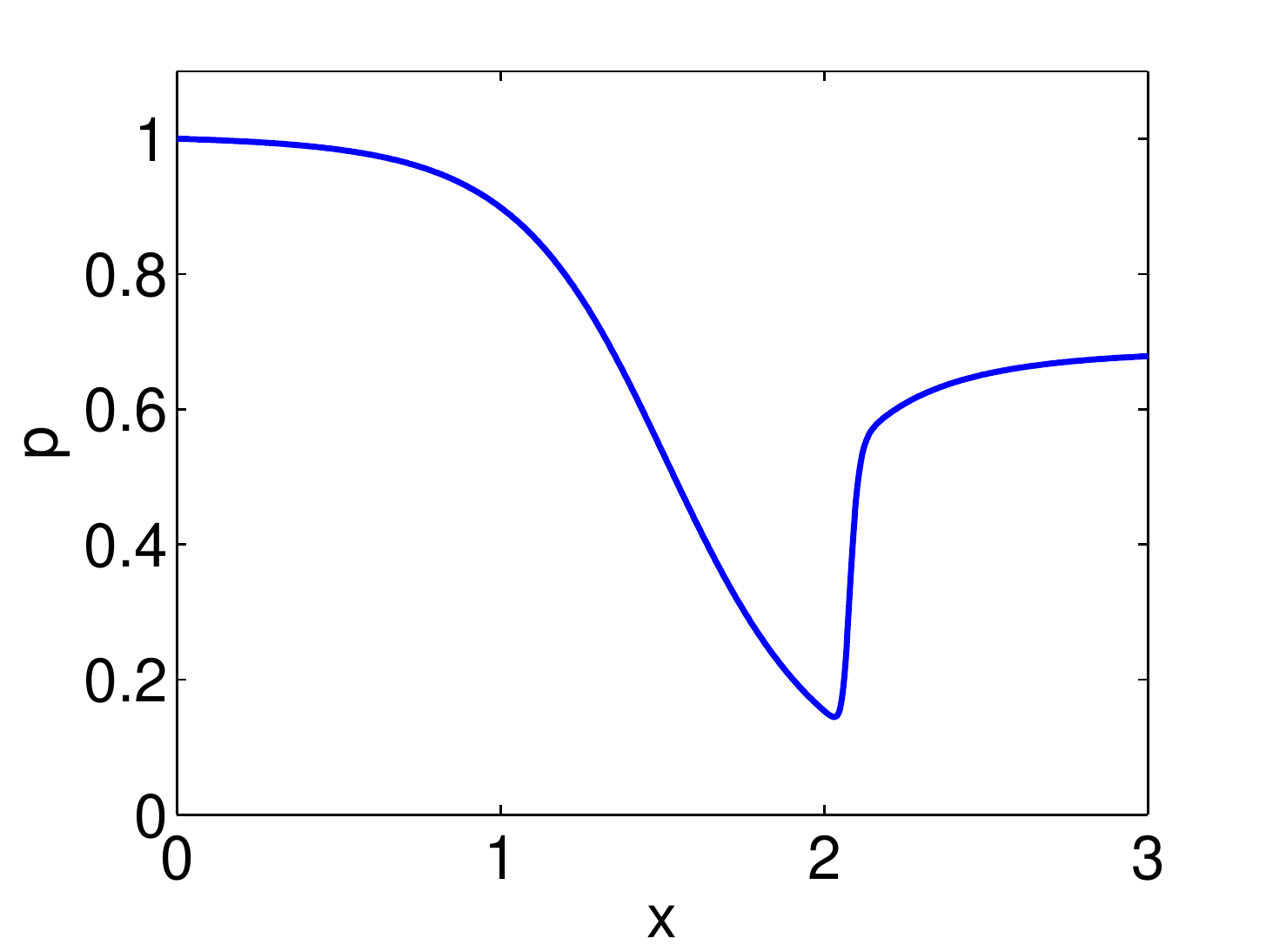}\label{fig:p_LF}}
       \subfigure[]{\includegraphics[width=.48\textwidth]{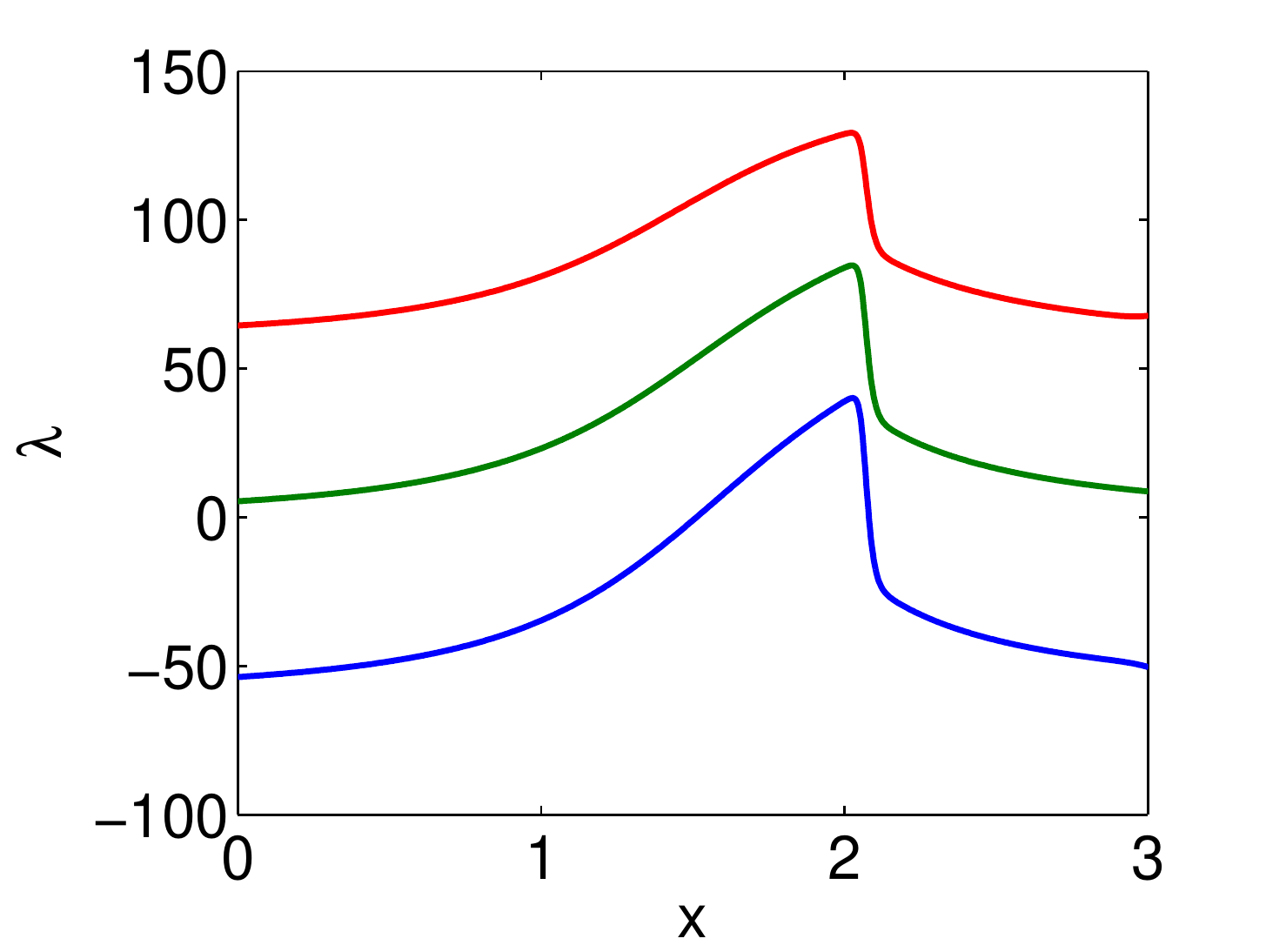}\label{fig:l_LF}}
  	\caption{Solution to the nozzle problem (pressure and eigenvalues) computed with 1000 grid points by \subref{fig:p_nozzle_1000},\subref{fig:l_nozzle_1000}~sweeping and \subref{fig:p_LF},\subref{fig:l_LF}~evolving a Lax-Friedrichs scheme to steady state.}
  	\label{fig:nozzle}  	
\end{figure} 

\begin{table}[htdp]
\caption{Computation time using $N$ grid points for the nozzle problem in~\S\ref{sec:nozzle}.}
\begin{center}
\begin{tabular}{c||cccccc}
N  & 64 & 128 & 256 & 512 & 1024 & 2048 \\
\hline
CPU Time (s) & 0.7 & 1.2 & 2.9 & 6.9 & 14.1 & 32.4\\
\end{tabular}
\end{center}
\label{table:nozzle}
\end{table}

\subsection{General Structure of Solutions Containing a Single Lax Shock}\label{sec:structure}
We have used several examples to illustrate the key ideas that are present in our fast sweeping approach.  Now we present a more systematic look at the general structure of stationary solutions to one-dimensional systems of conservation laws.

In the following discussion, we suppose that we are constructing the solution by sweeping from left to right.  Naturally, the opposite sweeping direction could be handled in a similar way.

We assume that on the left boundary, the first $I$ eigenvalues are negative, while on the right boundary, the first $J$ eigenvalues are negative.
\[ \lambda_1^L < \cdots < \lambda_I^L < 0 < \cdots < \lambda_n^L. \]
\[ \lambda_1^R < \cdots < \lambda_J^R < 0 < \cdots < \lambda_n^R. \]
We also suppose that the eigenvalues are all distinct,
\[ \lambda_i \neq \lambda_j, \quad \text{if }i \neq j. \]
Referring back to~\S\ref{sec:assumptions}, this set-up means that we have $I$ degrees of freedom on the left boundary,
\[ U = U_L^{\alpha_1,\ldots,\alpha_I}, \quad x = x_L \]
and $J$ conditions given at the right boundary,
\[ B_R(U) = \left(\begin{tabular}{c}$B_R^1(U)$\\ \vdots \\$B_R^J(U)$\end{tabular}\right) = 0, \quad x = x_R. \]
If $I\neq J$, then as we move from left to right, some of the eigenvalues will necessarily change sign.  This can happen in one of two ways:
\begin{enumerate}
\item Through a shock ($I<J$): This is the case if the $k^{th}$ eigenvalue is transitioning from positive to negative.  In this situation,  the unknown shock location $x_{S_k}$ is to be determined so that the solution matches the right boundary condition $B_R^k(U) = 0$.  At any point, the entropy conditions~\eqref{eq:entropy} ensure that only the smallest positive eigenvalue can have a shock.
\item Through a sonic (turning) point ($I>J$): This is the case if the $k^{th}$ eigenvalue is transitioning from negative to positive.  The source term must satisfy a compatibility condition for this to be possible.  In this situation, we are missing the boundary condition corresponding to this characteristic field (that is, there is an unknown parameter $\alpha_k$), but it is determined by the compatibility condition~\eqref{eq:compatibility} at the turning point $x_{T_k}$.  Since the solution is continuous through a turning point, only the largest negative eigenvalue can change sign through a turning point.
\end{enumerate}
We make a couple other observations.
\begin{enumerate}
\item The first $K \equiv min\{I,J\}$ degrees of freedom $(\alpha_1, \ldots, \alpha_K)$ may not be determined by a sonic point since the corresponding eigenvalues do not necessarily change sign in the domain.  Instead, these can be determined by the first $K$ components of the right boundary condition,
\[ B_R^1(U) = \ldots = B_R^K(U) = 0. \]
\item It is also possible for one of the other eigenvalues to change sign, as long as it changes back again.  For a positive eigenvalue, we would have an unknown shock condition determined by the compatibility condition at a subsequent sonic point.  For a negative eigenvalue, we would have an unknown left boundary condition, which is determined by the compatibility condition, followed by an unknown shock location, which is determined by the right boundary condition. 
\end{enumerate} 

We take a look at the structures required for different combinations of boundary conditions in order to obtain solutions that are continuous or have a single shock.  Similar reasoning can be used to examine the allowed structures for problems with multiple shocks.

The following discussion is quite general.  However, we note that in many cases it is possible to simplify these situations by using extra information about the problem.  For example, in the nozzle problem, the source term can only vanish at certain points that can be determined \emph{a priori} from the nozzle geometry, and sonic points are only possible at these points.  In other scenarios, physical intuition can limit the types of solutions we need to look for.

\subsubsection{Continuous solutions}

First we look at the structure required for continuous solutions.

{\bf{Case 1}}: $I<J$
\[ \lambda_1^L < \cdots < \lambda_I^L < 0 < \cdots < \lambda_J^L < \cdots < \lambda_n^L. \]
Now we see that $\lambda_{I+1},\ldots\lambda_J$ need to transition from positive to negative.  We expect that in general, this cannot be done continuously.

{\bf{Case 2}}: $I \geq J$
\[ \lambda_1^L < \cdots< \lambda_J^L \leq \cdots \leq \lambda_I^L < 0 < \cdots < \lambda_n^L. \]
In this case, $\lambda_{J+1},\ldots\lambda_I$ need to transition from negative to positive via sonic points in order from the largest to smallest eigenvalue.

Thus we will have $I$ unknowns in the form of missing boundary conditions on the left, and these will be determined by $J$ boundary conditions at right together with compatibility conditions for the turning points $x_{T_I}, x_{T_{I-1}},\ldots,x_{T_{J+1}}$.

\subsubsection{Solutions with a single shock}\label{sec:1shock}

Now we turn our attention to solutions that contain a single shock.

{\bf{Case 1}}: $I < J-1$.

In this case, we expect more than one shock using the same reasoning as Case~1 for continuous solutions.

{\bf{Case 2}}: $I = J-1$

\[ \lambda_1^L < \cdots < \lambda_I^L<0<\lambda_J^L < \cdots < \lambda_n^L. \]
Here $\lambda_J$ will transition from positive to negative via a shock.  The unknowns are $I$ left boundary conditions and one shock location.  These are determined by the $J = I+1$ right boundary conditions.  This structure is picture in Figure~\ref{fig:shock_case2}.

\begin{figure}[htdp]
	\centering
	{\includegraphics[width=0.45\textwidth,trim=0 40 0 0,clip=true]{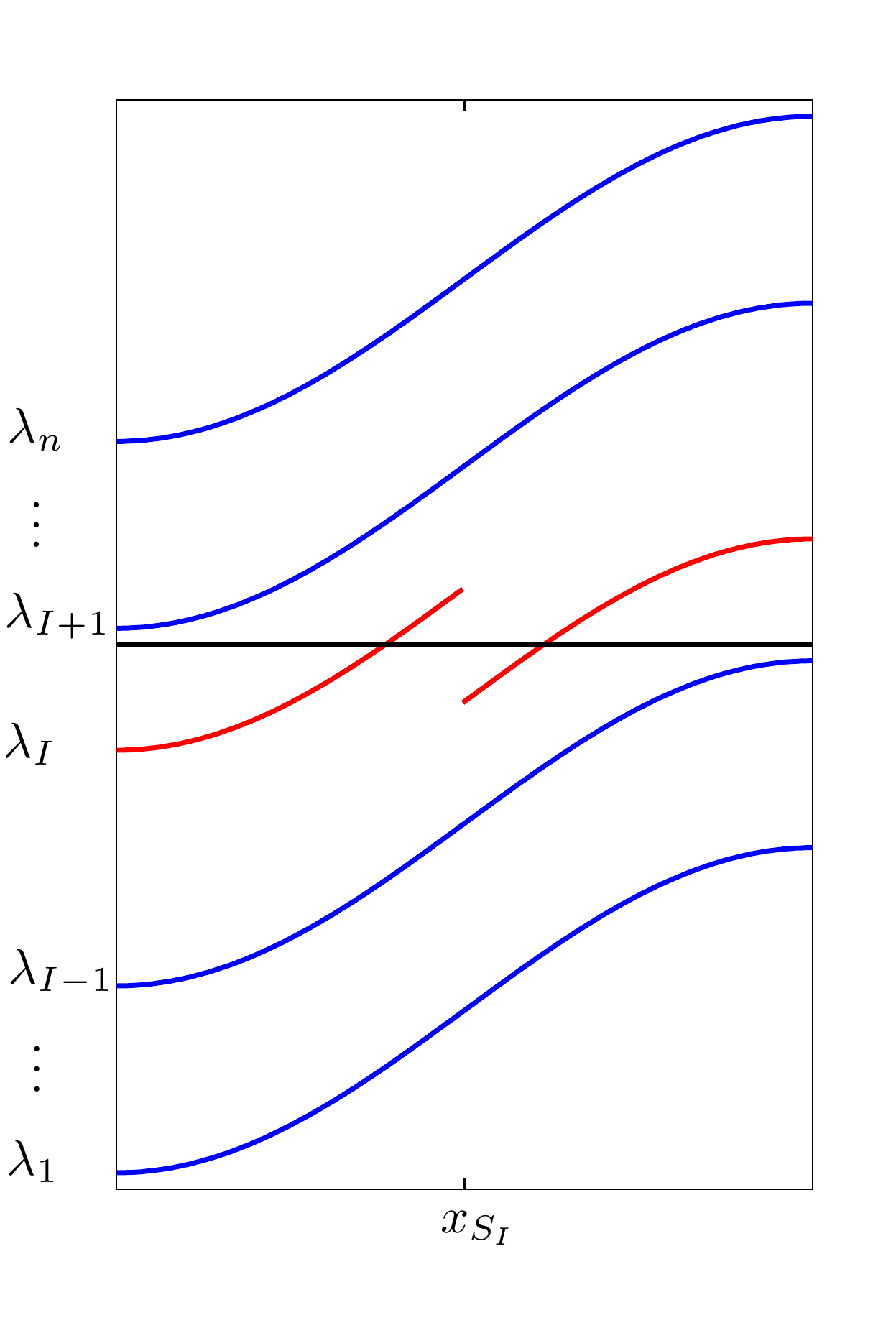}\label{fig:shock_case2}}
  	\caption{Structure of solutions from~\S\ref{sec:1shock} with $I=J-1$.}
  	\label{fig:shock_case2}  	
\end{figure} 

{\bf{Case 3}}: $I > J-1$

\[ \lambda_1^L < \cdots < \lambda_J^L\leq\ldots\leq\lambda_I^L<0<\cdots < \lambda_n^L. \]

We will require $\lambda_{J+1},\ldots, \lambda_I$ to transition from negative to positive via sonic points; these occur in order from largest to smallest eigenvalue.

On top of this basic structure, we want to introduce a shock.  We could introduce it at the far left, in $\lambda_{I+1}$, then follow it by a turning point in this same characteristic.

We could introduce the shock after the turning point $x_{T_k}$ ($I \geq k \geq J+1$):  a shock in $\lambda_k$, followed by another turning point in this field.

Finally, after the last necessary turning point $x_{T_{J+1}}$, we could introduce one more turning point $x_{T_J}$ and follow it by a shock $x_{S_J}$.

In each situation, the shock locations and missing left boundary conditions are the unknowns.  The sonic point compatibility conditions and the right boundary conditions are the equations that determine these unknowns.

For a visualisation of these permissible structure, see Figure~\ref{fig:1dstructure}.

\begin{figure}[htdp]
	\centering
	\subfigure[]{\includegraphics[width=0.45\textwidth]{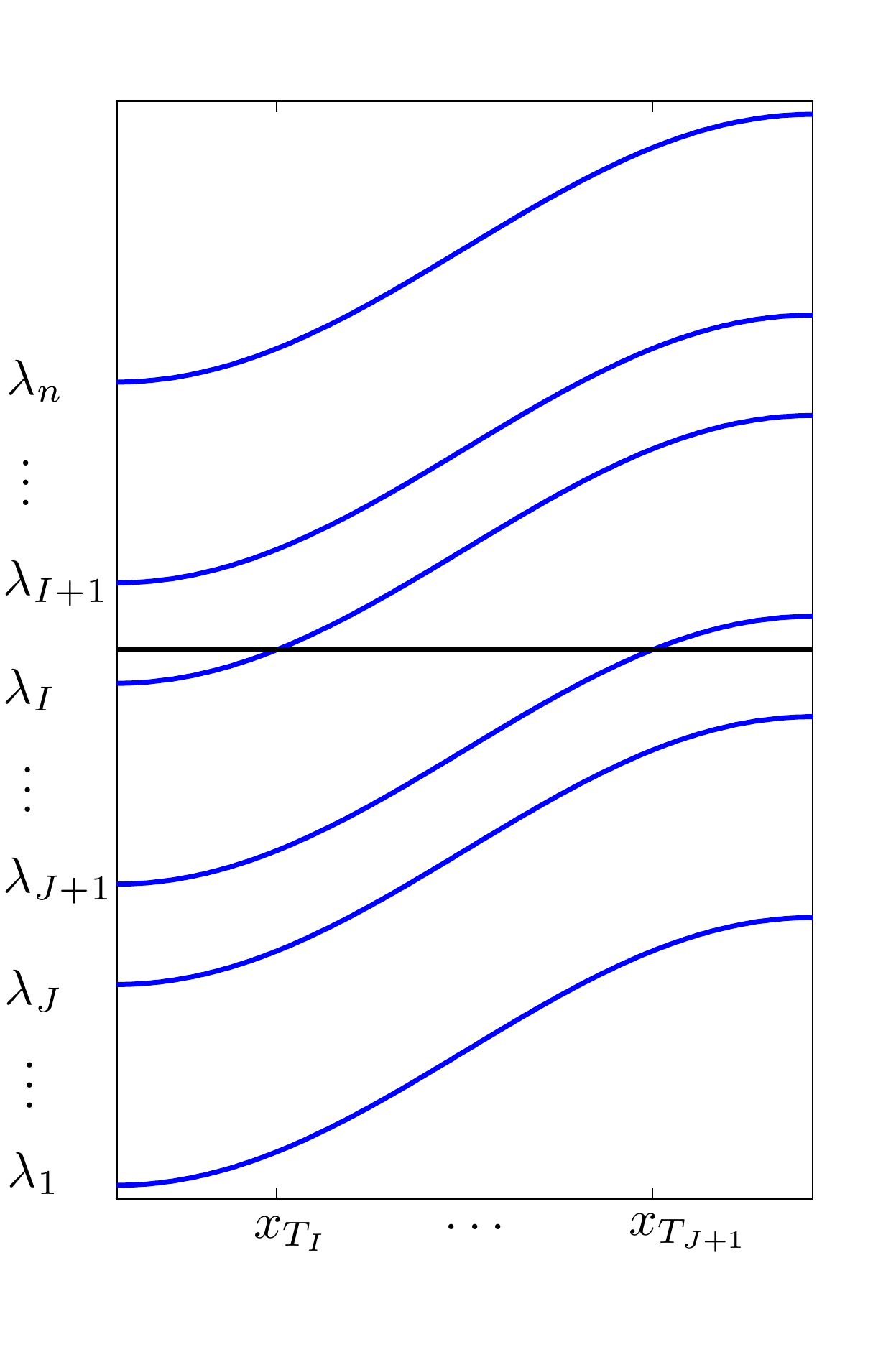}\label{fig:shock_0}}
	\subfigure[]{\includegraphics[width=0.45\textwidth]{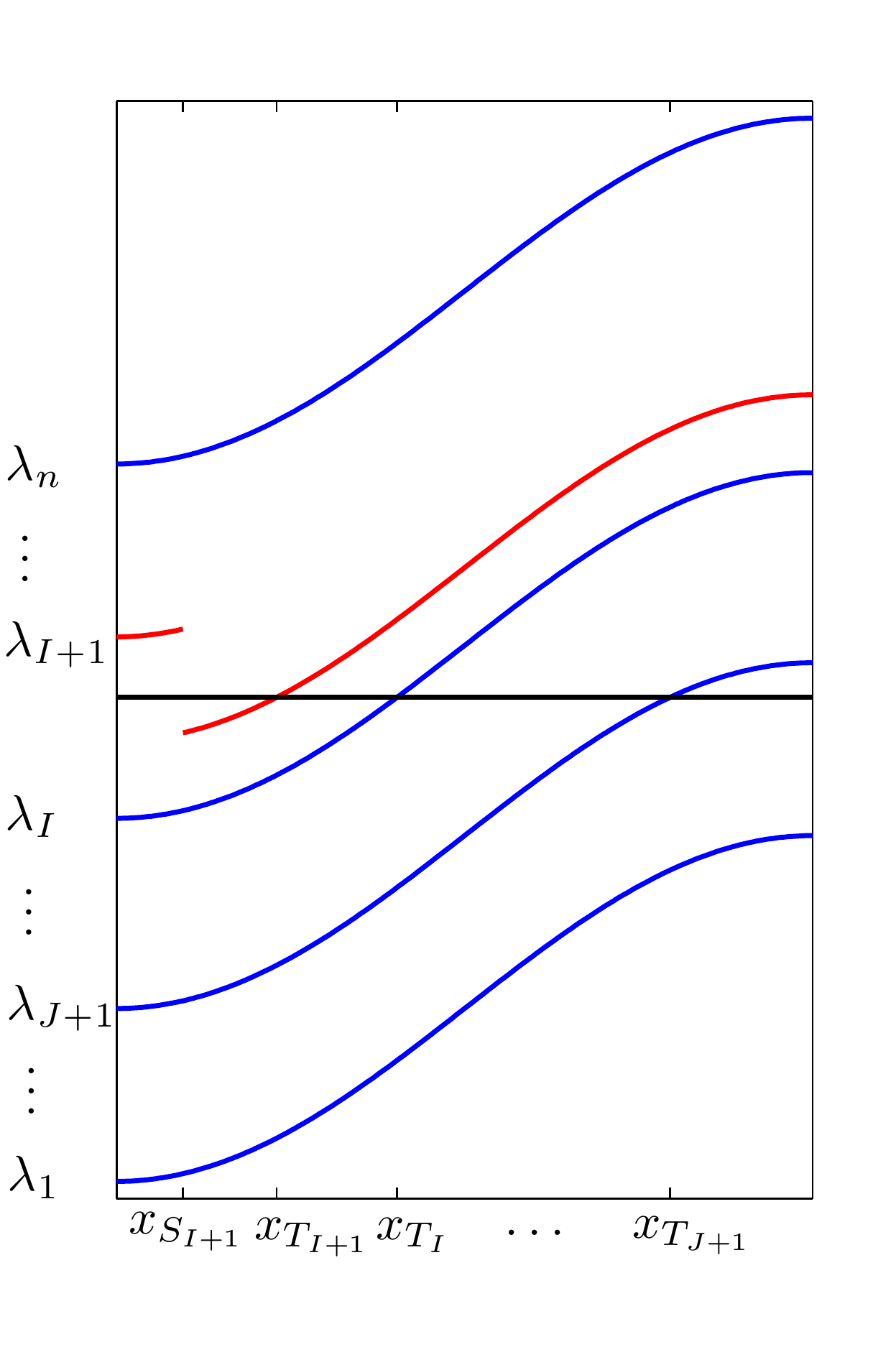}\label{fig:shock_I}}
	\subfigure[]{\includegraphics[width=0.45\textwidth]{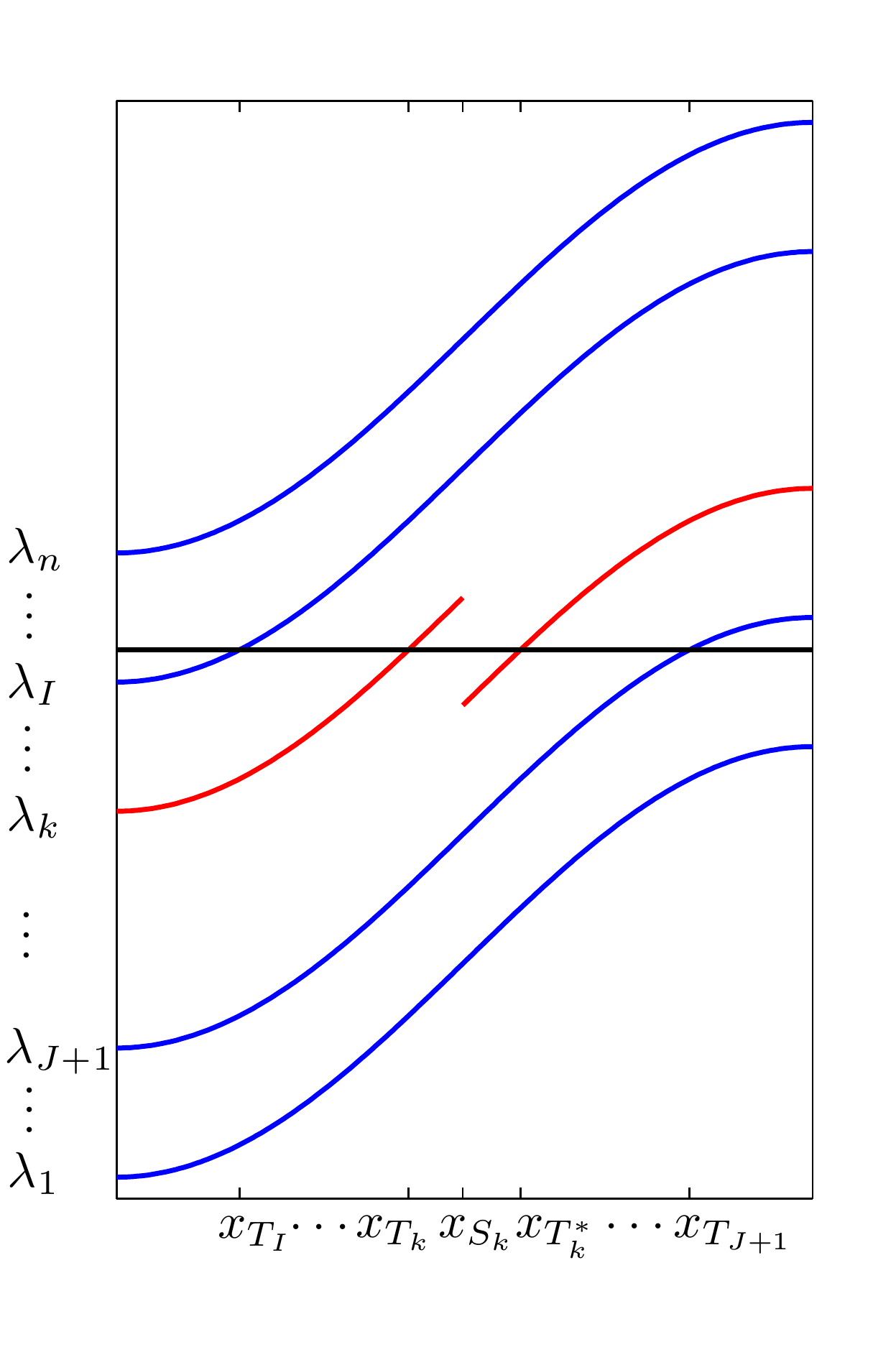}\label{fig:shock_k}}
	\subfigure[]{\includegraphics[width=0.45\textwidth]{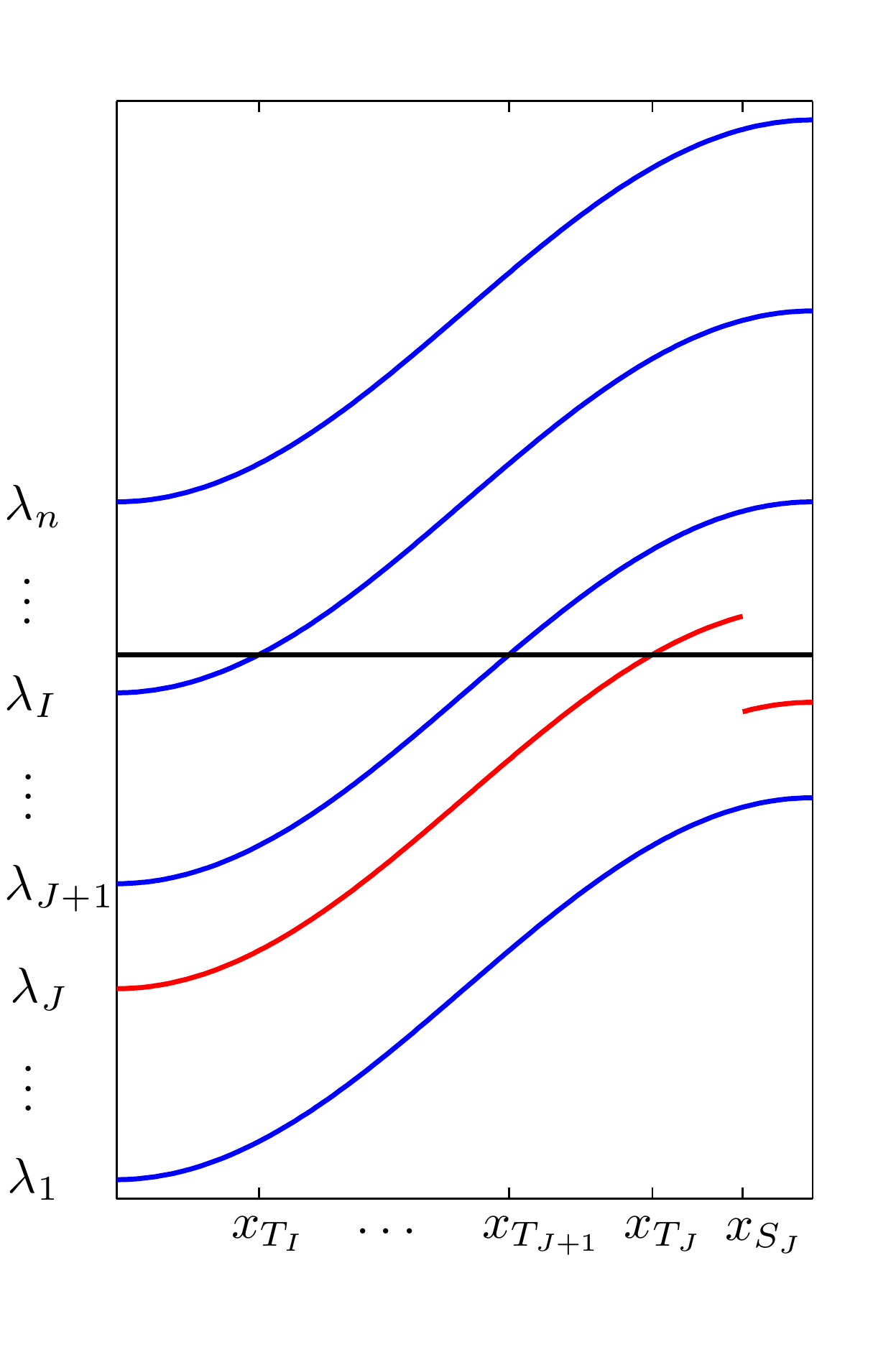}\label{fig:shock_J}}
  	\caption{Possible structures of solutions from~\S\ref{sec:structure} with $I>J-1$.  \subref{fig:shock_0}~Continuous solution, \subref{fig:shock_I}~shock in characteristic field $\lambda_{I+1}$, \subref{fig:shock_k}~shock in characteristic field $\lambda_k$ for $I \geq k \geq J+1$, and \subref{fig:shock_J}~shock in characteristic field $\lambda_J$.}
  	\label{fig:1dstructure}  	
\end{figure} 

Using this information about the permitted structure of solutions, we can now suggest a general algorithm for constructing a solution of~\eqref{eq:1dsteady} with a single, uniquely determined shock, subject to boundary conditions satisfying the assumptions of~\S\ref{sec:assumptions}.  By adjusting the initialisation of the unknowns, problems with multiple steady states could also be solved using this algorithm.

This approach can be founded upon any reasonable ODE solver for sweeping in (full) boundary conditions from the left ($x=x_L$).

Using the preceding discussion, and possibly additional information coming from physical intuition, we can limit the characteristic fields in which the shock can occur.  For each of these fields, we can attempt to construct a solution using Algorithm~\ref{alg:sweep1D}.

This algorithm requires solving a nested sequence of scalar equations.  In our implementation, we use a bisection method to solve these equations, but other solvers are also possible.  The unknowns that need to be determined are the missing boundary conditions at $x_L$ and the shock location.  Each unknown is determined by a matching condition---either a boundary condition at $x_R$ or a compatibility condition at a turning point.  These are summarised in Table~\ref{table:matching}.  Essentially, this method involves proceeding from left to right through the domain and computing each unknown in the order that its matching function is encountered.

Once these unknowns have been determined, a solution can be constructed by solving an initial value problem from $x_L$ to $x_S$, computing the appropriate jump at the shock, then solving another initial value problem form $x_S$ to $x_R$.

\begin{table}[htdp]
\caption{An overview of the structure of solution described in~\S\ref{sec:1shock} and Figures~\ref{fig:shock_case2}-\ref{fig:1dstructure} including each unknown, the matching function used to determine the unknown, the location where this matching occurs, and any conditions necessary for the presence of the unknown.}
\begin{center}
\begin{tabular}{|cccc|}
\hline
Unknown & Matching Function& Location for & Conditions \\
        &                  & Matching  & \\
\hline
$\alpha_I$      & $F_I = (P^{-1}a)_I$                  & $x_{T_I}$           & \multirow{3}{*}{$I\geq k$}\\
$\vdots$        & $\vdots $                            & $\vdots$            & \\
$\alpha_k$      & $F_k = (P^{-1}a)_k$                  & $x_{T_k}$           & \\
\hline
$x_{S_k}$       & $F_{k^*} = 
                   \begin{cases}
                   (P^{-1}a)_{k^*}, & k>J\\
                   B_R^k ,          & k=J
                   \end{cases}$                        & $\begin{cases}
                                                          x_{T_k^*}, & k>J\\
                                                          x_R,       & k=J
                                                          \end{cases}$       & \\
\hline                                                          
$\alpha_{k-1}$  & $F_{k-1} = (P^{-1}a)_{k-1}$          & $x_{T_{k-1}}$       & \multirow{3}{*}{$k > J+1$}\\
$\vdots$        & $\vdots$                             & $\vdots$            & \\
$\alpha_{J+1}$  & $F_{J+1} = (P^{-1}a)_{J+1}$          & $x_{T_{J+1}}$       & \\
\hline
$\alpha_J$      & $F_J = B_R^J$                        & $x_R$               & $k>J$\\
\hline
$\alpha_{J-1}$  & $F_{J-1} = B_R^{J-1}$                & $x_R$               & \multirow{3}{*}{$J>1$}\\
$\vdots$        & $\vdots$                             & $\vdots$            & \\
$\alpha_1$      & $F_1 = B_R^1$                        & $x_R$               & \\
\hline
\end{tabular}
\end{center}
\label{table:matching}
\end{table}

\algrenewcommand\algorithmicrepeat{\textbf{do}}
\algrenewcommand\algorithmicuntil{\textbf{while}}
\begin{algorithm}
\caption{Determine the unknowns summarised in Table~\ref{table:matching} in order to compute a solution with a single $k$-shock.}\label{alg:sweep1D}
\begin{algorithmic}[1]
\State Initialise $\alpha_1$
\Repeat $\quad\rhd$ Loop to determine $\alpha_1$
	\State$\vdots$
	\State Initialise $\alpha_{k-1}$
	\Repeat $\quad\rhd$ Loop to determine $\alpha_{k-1}$
		\State Initialise $\alpha_k$
		\Repeat $\quad\rhd$ Loop to determine $\alpha_k$
			\State$\vdots$
			\State Initialise $\alpha_I$
			\Repeat $\quad\rhd$ Loop to determine $\alpha_I$
				\State $\alpha_I\gets$ Update via nonlinear solver
			\Until $\abs{F_I(\alpha_1,\ldots,\alpha_I)} >\text{TOL}\,\,$
			\State $\vdots$
			\State$\alpha_k\gets$ Update via nonlinear solver
		\Until $\abs{F_k(\alpha_1,\ldots,\alpha_I)} >\text{TOL}\,\,$
		\State Initialise $x_{S_k}$
		\Repeat $\quad\rhd$ Loop to determine $x_{S_k}$
			\State $x_{S_k}\gets$ Update via nonlinear solver
		\Until $\abs{F_k^*(\alpha_1,\ldots,\alpha_I,x_{S_k})} >\text{TOL}\,\,$
		\State $\alpha_{k-1}\gets$ Update via nonlinear solver
	\Until $\abs{F_{k-1}(\alpha_1,\ldots,\alpha_I,x_{S_k})} >\text{TOL}\,\,$ 
	\State$\vdots$
	\State $\alpha_1\gets$ Update via nonlinear solver
\Until $\abs{F_1(\alpha_1,\ldots,\alpha_I,x_{S_k})} >\text{TOL}\,\,$
\end{algorithmic}
\end{algorithm}

\subsection{Convergence}\label{sec:converge}
In the special case where a solution consists of a single Lax shock with no turning points, we prove that our approach will compute the correct entropy solution.  Our methods also appear to compute the correct weak solution in the more general setting, but the well-posedness theory for these problems is much less clear, making it difficult to produce a very general proof.

In line with the discussion of the previous section, we express this steady state problem in the form
\bq\label{eq:CL1}
\begin{cases}
f(U)_x = a(U,x) & x_L < x < x_R\\
U = U_L^{\alpha_1,\ldots,\alpha_I} & x = x_L\\
B_R(U) = 0 & x = x_R.
\end{cases}
\eq

We say that this system is \emph{well-posed} if it satisfies the following assumptions.
\begin{enumerate}
\item[(A1)] The conservation law~\eqref{eq:CL1} has a unique solution $U^{ex}$, which consists of two smooth ($C^1$) states separated by a Lax shock at the location $x_S^{ex}$.  This solution is stable in $L^1$ under perturbations of the data.
\item[(A2)] For each $x\in[x_L,x_R]$, the flux function $f(U)$ is a $C^{1,1}$ diffeomorphism near $U^{ex}(x)$ and the source term $a(U,x)$ is Lipschitz near $(U^{ex}(x),x)$.  This ensures that the ODEs satisfied by each smooth solution component are well-posed.
\item[(A3)] There is a unique entropy satisfying solution of the Rankine-Hugoniot condition
\[ f\left(\Phi U^{ex}(x_S)\right) = f\left(U^{ex}(x_S)\right) \]
so that the jump operator is well-defined near $U^{ex}(x_S)$.
\item[(A4)] The functions $U_L^{\alpha_1,\ldots,\alpha_I}$ and $B_R(U)$ that define the boundary conditions are Lipschitz near $(\alpha_1^{ex},\ldots,\alpha_I^{ex})$ and $U^{ex}(x_R)$ respectively.
\end{enumerate}

We define the operator $\Prop_{x_1x_2}U_0$, which acts on an initial condition $U_0$ at a point $x_1$ by propagating it to $x_2$ via the solution of the system of ODEs
\bq\label{eq:ivp}
\begin{cases}
f(U)_x = a(U,x) & x_1 < x < x_2\\
U = U_0 & x = x_1
\end{cases}
\eq
so that $\Prop_{x_1x_2}U_0 = U(x_2)$.

We also recall that the jump operator $\Phi U_-$ returns a vector satisfying the Rankine-Hugoniot conditions
\[ f\left(\Phi U_-\right) = f(U_-) \]
as well as the Lax entropy conditions.

Using these operators, we can construct the solution to~\eqref{eq:CL1} if we are given the correct values of the unknowns outlined in Table~\ref{table:matching}:
\bq\label{eq:params}
y = \left(\begin{tabular}{c}$\alpha_1$\\ $\vdots$ \\ $\alpha_I$\\ $x_S$\end{tabular}\right).
\eq

Our approach then involves approximating the finite-dimensional solution vector $y^{ex}$ that satisfies
\bq\label{eq:CL_op} \G(y) \equiv B_R\left(\Prop_{x_Sx_R}\Phi\Prop_{x_Lx_S}U_L^{\alpha_1,\ldots,\alpha_I}\right) = 0. \eq
We approximate this by finding the solution $y^h$ of the discretised problem
\bq\label{eq:discrete_op} \G^h(y) \equiv B_R(\Prop^h_{x_Sx_R}\Phi^h\Prop^h_{x_Lx_S}U_L^{\alpha_1,\ldots,\alpha_I}) = 0, \eq
which is obtained by replacing the propagation and jump operators by discrete approximations.

Then the computed solution can be expressed as
\bq\label{eq:computed_sol}
U^h(x) = \begin{cases}
\Prop_{x_Lx}U_L^{\alpha_1^h,\ldots,\alpha_I^h} & x_L < x \leq x_S^h\\
\Prop_{x_S^hx}\Phi\Prop_{x_Lx_S^h}U_L^{\alpha_1^h,\ldots,\alpha_I^h} & x_S^h<x\leq x_R.
\end{cases}
\eq

\begin{remark}
In the special case of a one-dimensional scalar problem with boundary conditions given at $x_L$ and $x_R$, a simpler approach is to first compute left and right solution branches $U_-,U_+$; see~\S\ref{sec:RH_1d}.  In this case the only unknown is the shock location $x_S$ and the convergence results in Theorems~\ref{thm:converge}-\ref{thm:unique} can be applied to the scalar equation
\[ \G(x_S) \equiv f(U_-(x_S)) - f(U_+(x_S)) = 0. \]
\end{remark}

\begin{theorem}[Existence of a discrete solution]\label{thm:converge}
Suppose that the nonlinear conservation law~\eqref{eq:CL1} is well-posed.  Suppose also that the discrete operators $\Prop^h,\Phi^h$ are based upon consistent and stable approximations of the ODEs~\eqref{eq:ivp}, Lipschitz in the unknowns $y$, and approximate the continuous operators with accuracy on the order of $h^k$.  Then the discrete problem~\eqref{eq:discrete_op},\eqref{eq:computed_sol} has a solution $U^h$. Moreover, there is a constant $C$ such that for sufficiently small $h$, each component of the discrete solution satisfies
\[ \|u^{ex} - u^h\|_{L^1(x_L,x_R)} \leq Ch^k. \]
\end{theorem}

\begin{proof}
We begin by using assumptions (A1)-(A4) to make several observations about the operator $\G(y)$:
\begin{enumerate}
\item[(B1)] The equation
\[ \G(y) = 0 \]
has a unique solution $y^{ex}$.
\item[(B2)] There exists an open set $V$ containing the origin such that the inverse operator $\G^{-1}$ is defined and Lipschitz in $V$.  In particular, for every $b\in V$,
\[ \|\G^{-1}(b) - y^{ex}\| \leq C\|b\|. \]
\item[(B3)] There exists an open set $Y$ containing $y^{ex}$ such that $\G:Y\to V$ is Lipschitz continuous.
\end{enumerate}

Next we define the operator
\bq\label{eq:new_op} \F^h(y) \equiv \G^{-1}\left(\G(y)-\G^h(y)\right), \quad y\in Y, \eq
which is defined for sufficiently small $h$ as a consequence of~(B2). We can also say that
\[ \|\F^h(y)-y^{ex}\| = \|\G^{-1}\left(\G(y)-\G^h(y)\right)-y^{ex}\| \leq C \|\G(y)-\G^h(y)\| \leq Ch^k. \]
We can conclude that the range of $\F^h$ is contained in a ball of radius $Ch^k$ centred at $y^{ex}$:
\[ \F^h : Y \to B(y^{ex}; Ch^k) \]
Additionally, we know that for sufficiently small $h$, this ball is contained in the set $Y$,
\[  B(y^{ex}; Ch^k) \subset Y \]
so that
\[\F^h:B(y^{ex}; Ch^k)\to B(y^{ex}; Ch^k).\]

Since $\F^h$ is a continuous operator (by continuity of $\G,\G^{-1},\G^h$), we can use Brouwer's fixed point theorem to conclude that $\F^h$ has a fixed point $y^h$ in this ball.  That is, there exists $y^h\in B(y^{ex}; Ch^k)$ such that
\[ y^h = \F^h(y^h) = \G^{-1}\left(\G(y^h) - \G^h(y^h)\right), \]
which means that
\[ \G^h(y^h) = 0 \quad \text{and} \quad \|y^h-y^{ex}\| \leq Ch^k. \]
From the stability of the propagation and jump operators~(A2),(A3) and the accuracy of their discrete approximations, we conclude that the discrete solution $U^h$ has accuracy on the order of $h^k$ in $L^1$.
\end{proof}

We can also conclude that as long as we restrict the choice of parameters to what is essentially the regime where the conservation law~\eqref{eq:CL1} is well-posed, there is no danger that an appropriate discrete approximation will compute any spurious solutions that are far away from the correct entropy solution.

\begin{theorem}[Non-existence of spurious discrete solutions]\label{thm:unique}
Under the hypotheses of Theorem~\ref{thm:converge}, there is an open set $Y$, independent of $h$, such that any solution $y^h_*\in Y$ of the discrete problem~\eqref{eq:discrete_op} satisfies
\[ \|y^h_*-y^{ex}\| \leq Ch^k. \]
\end{theorem}

\begin{proof}
Let $Y$ be the open set defined in~(B3).  Using properties~(B1)-(B3), we conclude that
\begin{align*}
\|y^h_* - y^{ex}\|
  &= \|\G^{-1}(\G(y^h_*)) - y^{ex}\|\\
  &\leq C\|\G(y^h_*)\| \\
  &= C\|\G(y^h_*) - \G^h(y^h_*)\|\\
  &\leq Ch^k.
\end{align*}
Thus any solution of the discretised problem lying in the set $Y$ (roughly, the parameter regime where the original problem is well-posed) will approximate the solution of the exact problem with an accuracy on the order of $h^k$.
\end{proof}

\section{Two-Dimensional Problems}\label{sec:2d}
Next we turn our attention to steady state solutions of the two-dimensional conservation law,
\bq\label{eq:cl2D} u_t + f(u)_x + g(u)_y = a(u,x,y), \eq
together with suitable boundary conditions.

The idea of our approach is to view the steady state equations as a free boundary problem.  The solution will consist of smooth states that are separated by a shock curve as in~\cite{Feldman}.  Our approach to this problem involves two basic steps:
\begin{enumerate}
\item Computing the smooth solution branches.
\item Constructing the shock curve (that is, the free boundary) that separates the smooth states.
\end{enumerate}

\subsection{Generating Solution Branches}\label{sec:branch2d}
We start by generating solution branches by sweeping in the boundary conditions.  
To do this, we look at a paraxial form of the equation, which is essentially treating one of the spatial dimensions like a time dimension~\cite{QianSymesParaxial03}.

For example, if we want to sweep in the bottom boundary condition, we would treat the $y$-direction like the time axis and solve 
\bq\label{eq:paraxial}
v_y + f(g^{-1}(v))_x = a(g^{-1}(v),x,y) 
\eq
as long as locally we can invert $g(u)$.  Then the bottom branch of the solution is 
\[ u_B = g^{-1}(v). \]
This inversion can be accomplished efficiently using Newton's method, using the value from the previous ``time'' step as a starting guess.

We can perform this sweeping using any suitable method.  In the computations below, we use forward Euler for the ``time'' dimension and a Gudonov flux for the ``spatial'' dimension, but other methods---including higher-order or even non-conservative methods---can also be incorporated into this sweeping procedure. 

We may not be able to sweep all the way across the domain.  If at some point an eigenvalue of $\nabla g(u_B)$ becomes close to zero, we cannot sweep this value any farther.    This indicates that this vertical sweeping direction is not appropriate for updating this portion of the domain.  Instead, these values will be computed by sweeping from the left or right.  See the example in~\S\ref{sec:2dscalar}.

Also, if the given boundary data on the left and right are not consistent with the orientation of the characteristics (determined by the sign of the eigenvalues of $\nabla g$), we discard these values.  If possible (that is, if the characteristic structure permits it), we can instead update these boundary values using the given conservation law and appropriate one-sided differences.

We can use a similar procedure to generate solution branches that sweep from the other sides of the domain.

\subsection{Matching Solution Branches}\label{sec:match}
Once we obtained the required solution branches, we combine these two at a time to assemble the final solution.  For example, we can start with the left branch $u_L$ and the bottom branch $u_B$.  We construct a curve that splits the domain into two pieces, which determines which solution branch should be used where.  We start the curve at a discontinuity, where the given ``initial'' conditions for the two branches will meet.  As described below, we use the Rankine-Hugoniot conditions to extend this curve until it again hits the boundary of the domain.  We can then repeat the procedure using other solution branches until all boundary conditions are satisfied.

To grow the curve that divides two solution branches, we will look at one small cell $[x_i,x_{i+1}]\times[y_j,y_{j+1}]$.  We know where the curve enters this cell and want to determine where the curve will exit this cell.  If we approximate the curve by a straight line segment in this cell, this exit point can be determined once we know the direction normal to the curve.

To compute the normal, we require an approximation of the jumps in flux ($[[f]]$, $[[g]]$) at the entry point.  These values are easily computed at nearby grid points via
\[ [[f]] = f(u_+) - f(u_-), \quad [[g]] = g(u_+) = g(u_-) \]
where $u_+$ and $u_-$ are the values of the solution branches that are being matched.  Once this is done, we can interpolate to approximate the change in flux at the entry point.

We typically expect the normal vector to satisfy the Rankine-Hugoniot condition~\eqref{eq:RH_2D}:
\[ n_1[[f]] + n_2[[g]] = 0. \]
We can always find a direction that satisfies this condition.  If the direction we come up with satisfies the entropy condition~\eqref{eq:entropy2D},
\[ n_1f'(u_-) + n_2g'(u_-) > 0, \quad n_1f'(u_+) + n_2g'(u_+) < 0,  \]
then we can extend the curve using this value.

It is worth noting that the procedure for matching two solution branches only requires $\bO(\sqrt{N})$ time, where $N$ is the total number of grid points; this has no effect on the overall computational complexity of the algorithm since the sweeping step requires $\bO(N)$ time.

\subsubsection{Example with three states}\label{sec:3states}
In the first example, we consider the equation
\bq\label{eq:3states} \left(ku^2\right)_x + \left(u-u^3\right)_y = 0, \quad [x,y]\in[0,1]^2. \eq

We choose boundary conditions from three different constant values: $u_0 = 0$ on the bottom, $1/\sqrt{3}<u_L<1$ on the left side, as well as the left half of the top side, and $uR = -uL$ on the remainder of the boundary.  

The exact solution consists of three constant states divided by straight line segments.  The line segment joining the bottom and left states starts from the lower-left corner and has slope $\alpha = \frac{1-u_L^2}{ku_L}$.  Similarly, the line segment joining the bottom and right states has slope $-\alpha$.

We use our sweeping approach to compute this solution,  taking $\alpha=1.2$ and $u_L = 0.75$.  This involves first combining the left and bottom states (starting from the bottom left corner), then combining this result with the right state (starting from the bottom right corner).  

We also repeat this example, this time replacing the left and right boundary values by
\[  u(0,y) = 0.75+0.2\sin(\pi y), \quad u(1,y) = -u(0,y). \]
The resulting solution will now consist of three non-constant states, which are divided by curves rather than straight line segments.

The computed solutions for both examples are shown in \autoref{fig:3states}.  We make particular note of the sharp shocks that were produced with this method.  Computation times are given in Table~\ref{table:3states}.  For comparison, we also provide computation times for a simple explicit time-stepping method using Godunov fluxes.  It is clear that the sweeping method is much more efficient.

\begin{table}[htdp]
\caption{Computation time on an  $m\times m$ grid ($N=m^2$) for the examples with three states in~\S\ref{sec:3states}.}
\begin{center}
\begin{tabular}{c||cccccc}
& \multicolumn{6}{c}{Constant States}\\
\hline
$m$  & 32 & 64 & 128 & 256 & 512 & 1024\\
CPU Time (s) for Sweeping & 0.15 & 0.28 & 0.70 & 2.06 & 6.78 & 24.41\\
CPU Time (s) for Evolution & 0.07 & 0.17 & 0.85 & 7.97 & 127.20 & ---\\
& \multicolumn{6}{c}{Non-constant States}\\
\hline
$m$  & 32 & 64 & 128 & 256 & 512 & 1024\\
CPU Time(s) for Sweeping & 0.18 & 0.38 & 0.92 & 2.57 & 8.20 & 28.68\\
CPU Time (s) for Evolution & 0.07 & 0.18 & 0.74 & 7.28 & 115.99&---\\
\end{tabular}
\end{center}
\label{table:3states}
\end{table}

\begin{figure}[htdp]
	\centering
	\subfigure[]{\includegraphics[width=.48\textwidth]{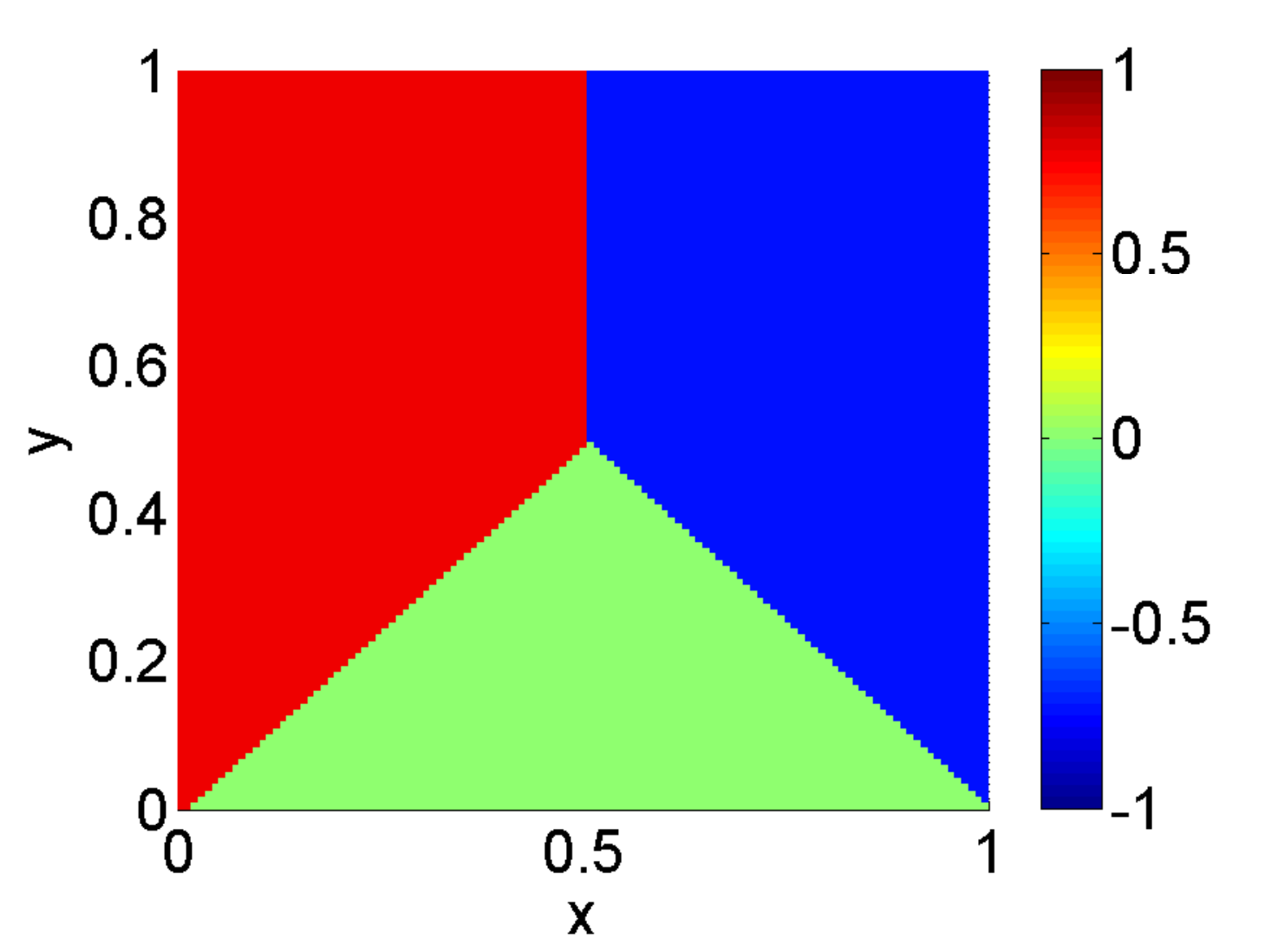}\label{fig:three_sol}}
        \subfigure[]{\includegraphics[width=.48\textwidth]{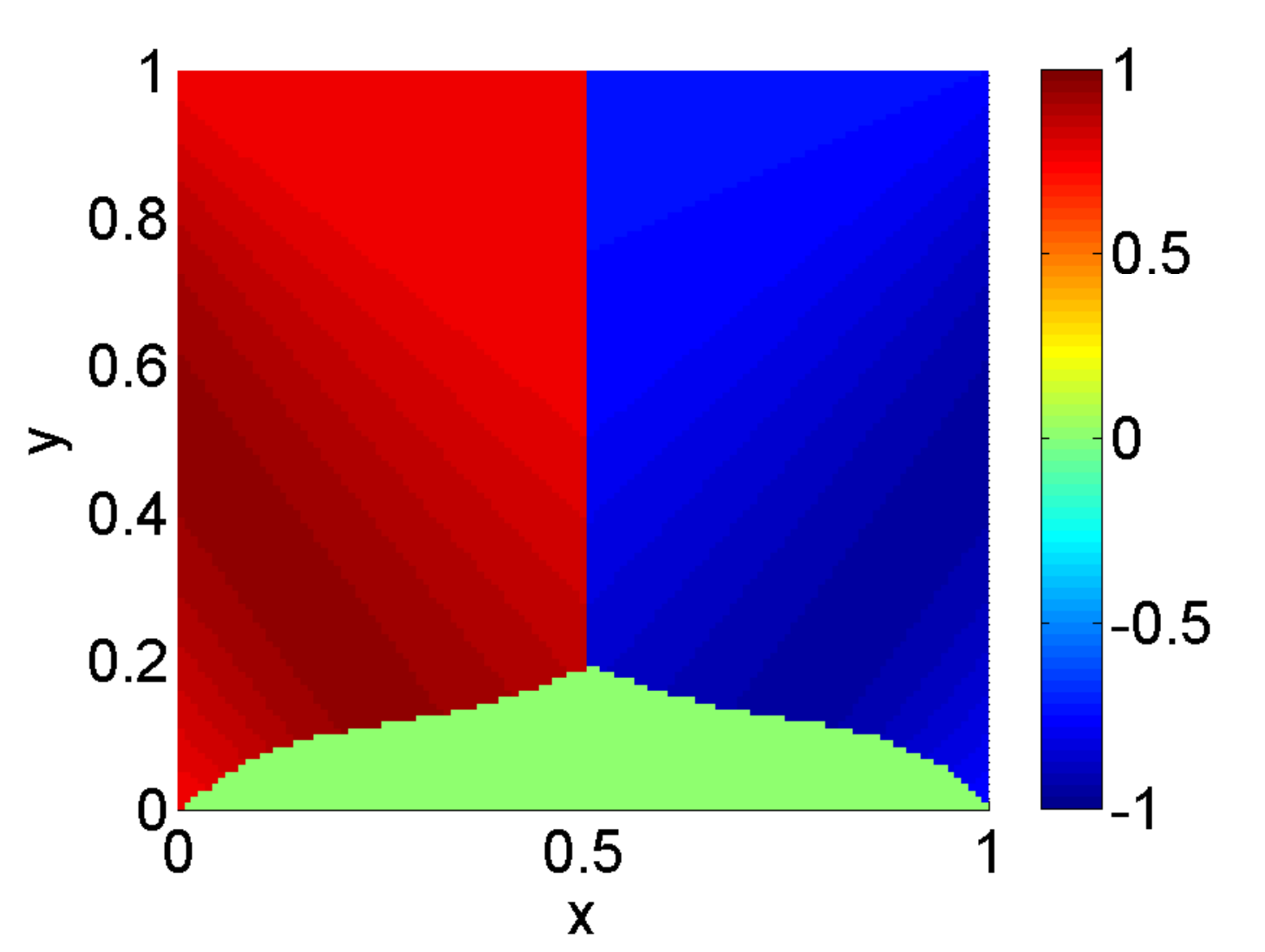}\label{fig:three_eps}}
  	\caption{Solutions from~\S\ref{sec:3states} with three \subref{fig:three_sol}~constant or \subref{fig:three_eps}~non-constant states computed on a $120\times120$ grid.}
  	\label{fig:3states}  	
\end{figure}

\subsection{Verifying the Entropy Condition}\label{sec:checkEntropy}
In certain degenerate cases, special care is needed in constructing the correct, entropy-satisfying curve.  For example, if the flux functions are the same: $f(u) = g(u)$ then the direction $n_1 = -n_2$ will always satisfy the Rankine-Hugoniot condition, 
\[ n_1[[f]] + n_2[[g]] = 0, \]
but it may not satisfy the entropy condition.  In this case, we instead need to find a direction $n$ that will make the change in flux zero across the curve: $[[f]] = [[g]] = 0$. To accomplish this, we choose a direction that will cause the curve to exit a side where the jumps $[[f]]$ and $[[g]]$ change sign, again checking the entropy condition.

This could still fail due to numerical errors introduced in the sweeping step.  It may be impossible to make the change in flux exactly zero across the curve.  Then we just need to make this change as small as possible in some sense.  For example, we could choose to have the curve exit the side that makes the quantity
\[ [[f_1]][[f_2]],[[g_1]][[g_2]]\} \]
as small as possible.  Here $[[f_1]]$ and $[[f_2]]$ are the jump in flux evaluated at two adjacent corners of the cell, which form the endpoints of one side of the cell.  This quantity is always positive since either $[[f]]$ or $[[g]]$ is not changing sign.  Again, we limit ourselves to directions that are entropy correct,
\[ n_1f'(u_-) + n_2g'(u_-) > 0, \quad n_1f'(u_+) + n_2g'(u_+) < 0.  \]

\subsubsection{2D Burger's equation}\label{sec:burgers2d}
The next example we consider is the two-dimensional Burger's equation, which illustrates the importance of verifying the entropy conditions since the direction $n_1 = -n_2$ will always satisfy the Rankine-Hugoniot condition, even though this does not lead to the correct solution.  Here we use the burger's flux and solve
\bq\label{eq:burgers2d} \left(\frac{u^2}{2}\right)_x + \left(\frac{u^2}{2}\right)_y = u(1-\phi'(x))\psi'(y-\phi(x)) \eq
with 
\[\phi(x) = 0.5+0.5\cos(\pi x), \quad \psi(z) = -\sin(\pi z), \quad  u_0^L = 2,\quad u_0^R = -2. \]

The exact solution consists of two smooth components separated by the curve $y = \phi(x)$:
\[u(x,y) = 
\begin{cases}
u_0^L + \psi(y-\phi(x)) & y<\phi(x)\\
u_0^R + \psi(y-\phi(x)) & y > \phi(x).
\end{cases}
\]
We solve this by sweeping and matching different solution branches; 
the computed and exact solutions are shown in Figure~\ref{fig:burgers2d}.  Computation times, shown in Table~\ref{table:burgers2d}, validate our claim that the computational complexity of this method is linear in the number of grid points.

\begin{table}[htdp]
\caption{Computation time on an  $m\times m$ grid ($N=m^2$) for the 2D Burger's equation in~\S\ref{sec:burgers2d}.}
\begin{center}
\begin{tabular}{c||cccccc}
$m$  & 32 & 64 & 128 & 256 & 512 & 1024\\
\hline
CPU Time (s) & 0.15 & 0.19 & 0.45 & 1.22 & 3.89 & 12.85\\
\end{tabular}
\end{center}
\label{table:burgers2d}
\end{table}

\begin{figure}[htdp]
	\centering
	\subfigure[]{\includegraphics[width=.48\textwidth]{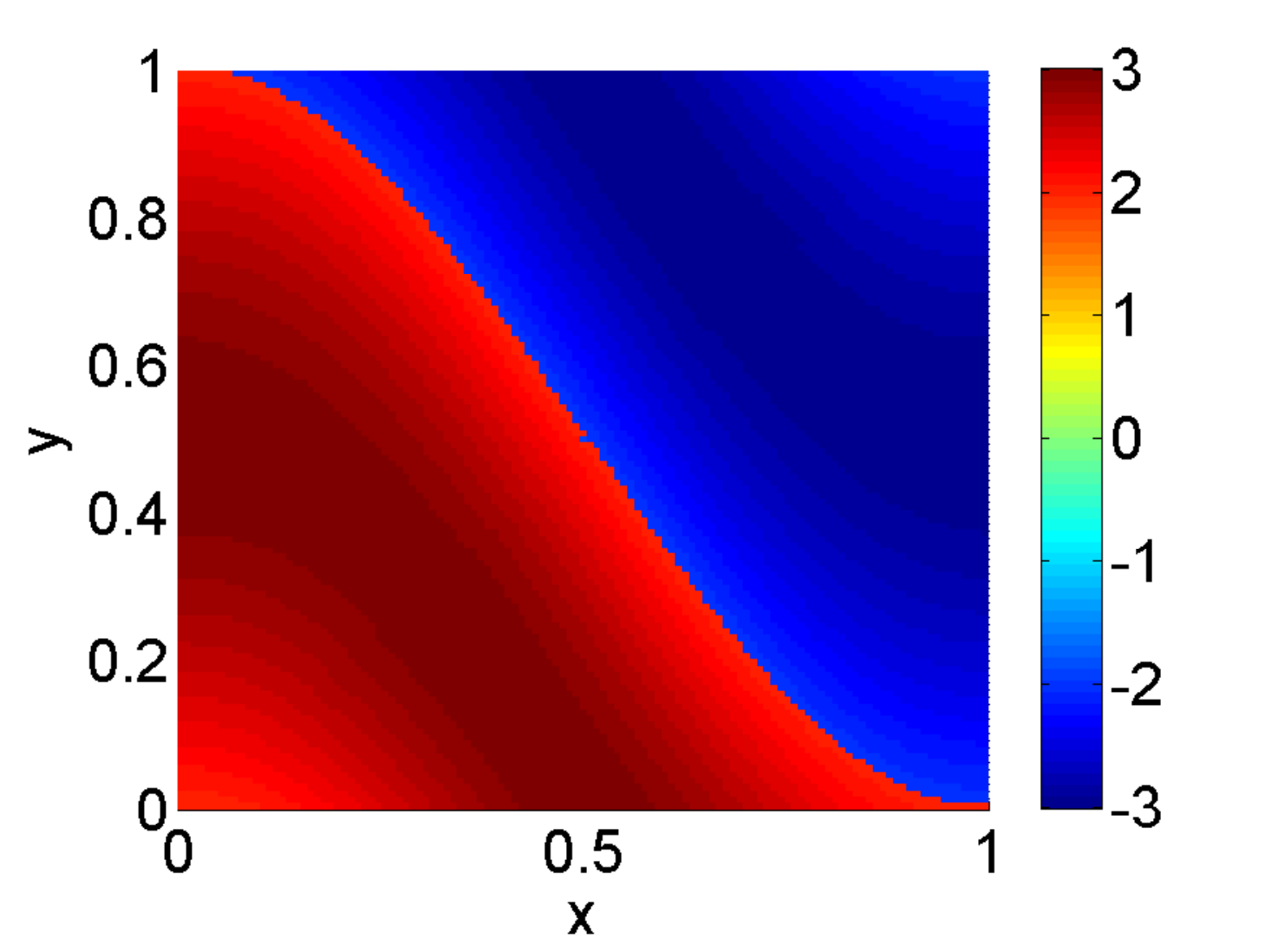}\label{fig:RH_2D_sol}}
  \subfigure[]{\includegraphics[width=.48\textwidth]{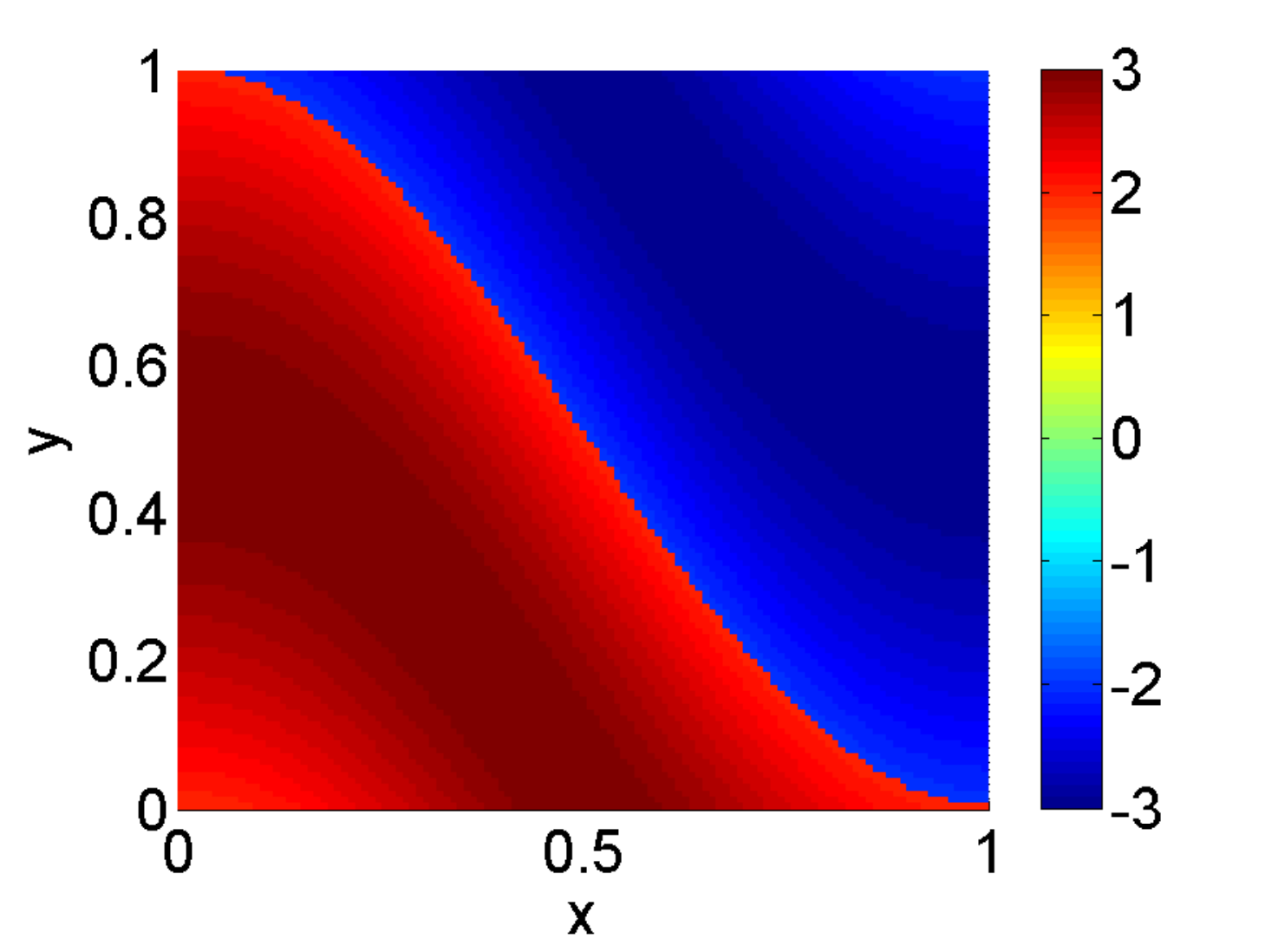}\label{fig:burgers_ex}}
  	\caption{\subref{fig:RH_2D_sol}~Computed and \subref{fig:burgers_ex}~exact solution to the 2D Burger's equation in~\S\ref{sec:burgers2d} on a $120\times120$ grid.}
  	\label{fig:burgers2d}  	
\end{figure}

\subsection{Sweeping through a Shock}\label{sec:sweepShock}

It is clear that as long as solutions are smooth, the paraxial form of the conservation law is equivalent to the original steady state equations.  However, in some cases a shock could develop as we evolve this paraxial equation.  We show that the shock curve that develops is a valid stationary shock of the original conservation law.  Then provided we use a conservative method to solve the paraxial equation, any resulting shocks will be valid entropy shocks.

\begin{theorem}[Equivalence of stationary conservation law and paraxial equation]
Let $u:\Omega\subset\R^2\to\R$ be a function consisting of two $C^1$ states separated by a smooth curve $\Gamma$, which divides the domain $\Omega$ into two disjoint sets $\Omega_-$ and $\Omega_+$.  Let $f$ and $g$ be two differential flux functions and assume that $g'(u)>0$ at all points in $u(\Omega)$.  Then $u$ is a stationary entropy solution of the conservation law~\eqref{eq:cl2D} if and only if $v \equiv g(u)$ is an entropy solution of the paraxial equation~\eqref{eq:paraxial}.
\end{theorem}

\begin{proof}
Suppose that $v=g(u)$ is an entropy solution of the paraxial equation.  We first show that it is also a stationary solution of the original conservation law.

If $x\notin\Gamma$ then $u$ is smooth at this point and the two formulations are trivially equivalent.

We now consider points $x\in\Gamma$ that lie on the shock curve.  We further let $(n_1,n_2)$ be a vector normal to the curve and pointing from $\Omega_-$ to $\Omega_+$.

The speed of the shock obtained from the paraxial form of the equation is
\[  s = \frac{[[f(g^{-1}(v))]]}{[[v]]} = \frac{[[f(u)]]}{[[g(u)]]}. \]
The normal vector is related to the shock speed through
\[  -\frac{n_2}{n_1} = s = \frac{[[f(u)]]}{[[g(u)]]}. \]
Rearranging, we find that
\[ (n_1,n_2)\cdot([[f(u)]],[[g(u)]])=0, \]
which is precisely the Rankine-Hugoniot condition for a stationary shock~\eqref{eq:RH_2D}.

We also look at the entropy condition for this solution of the paraxial equation:
\[ \left.\frac{d}{dv}f(g^{-1}(v))\right|_{v = v_+} < s < \left.\frac{d}{dv}f(g^{-1}(v))\right|_{v = v_-}. \]
This is equivalent to
\[ \frac{f'(u_+)}{g'(u_+)} < -\frac{n_2}{n_1} < \frac{f'(u_-)}{g'(u_-)}. \]
Rearranging, we find that
\[ (n_1,n_2)\cdot(f'(u_+),g'(u_+)) < 0 < (n_1,n_2)\cdot(f'(u_-),g'(u_-)), \]
which is the entropy condition~\eqref{eq:entropy2D} for a stationary shock solution of the original conservation law.

We conclude that the function $u$ is a stationary entropy solution of the conservation law.
Since all the above steps are reversible, this completes the proof.
\end{proof}

\subsubsection{Example where a shock forms}\label{sec:2dscalar}
We consider an example from~\cite{Chen_LFSweeping}, which involves solving
\bq\label{eq:2dscalar} \left(\frac{u^2}{2}\right)_x + u_y = 0 \eq
subject to the boundary conditions 
\[ u(0,y) = 1.5, \quad u(1,y) = -0.5, \quad u(x,0) = 1.5-2x. \]
In this example, we can obtain the entire solution by sweeping once from the bottom of the domain.  The computed solution is shown in Figure~\ref{fig:2dscalar}.  In addition, we include computation times (Table~\ref{table:2dscalar}) to demonstrate the linear computational complexity of the sweeping step.

For illustration of the effects of sweeping, we also show the left solution branch, which is only defined up to the region where $f'(u)$ vanishes.  Note that this is not needed to generate this solution.  However, we could choose to match these two branches using the procedure described in~\S\ref{sec:match}; this would produce a sharp shock instead of a shock spread over a couple grid points (as would be computed by conventional methods).

\begin{table}[htdp]
\caption{Computation time on an $m\times m$ grid ($N=m^2$) for the 2D scalar equation in~\S\ref{sec:2dscalar} solved by sweeping once from the bottom, or by sweeping and matching the bottom and left solution branches.}
\begin{center}
\begin{tabular}{c||cccccc}
$m$  & 32 & 64 & 128 & 256 & 512 & 1024\\
\hline
CPU Time (s) & 0.05 & 0.11 & 0.22 & 0.48 & 1.14 & 3.08\\
\end{tabular}
\end{center}
\label{table:2dscalar}
\end{table}

\begin{figure}[htdp]
	\centering
	\subfigure[]{\includegraphics[width=.48\textwidth]{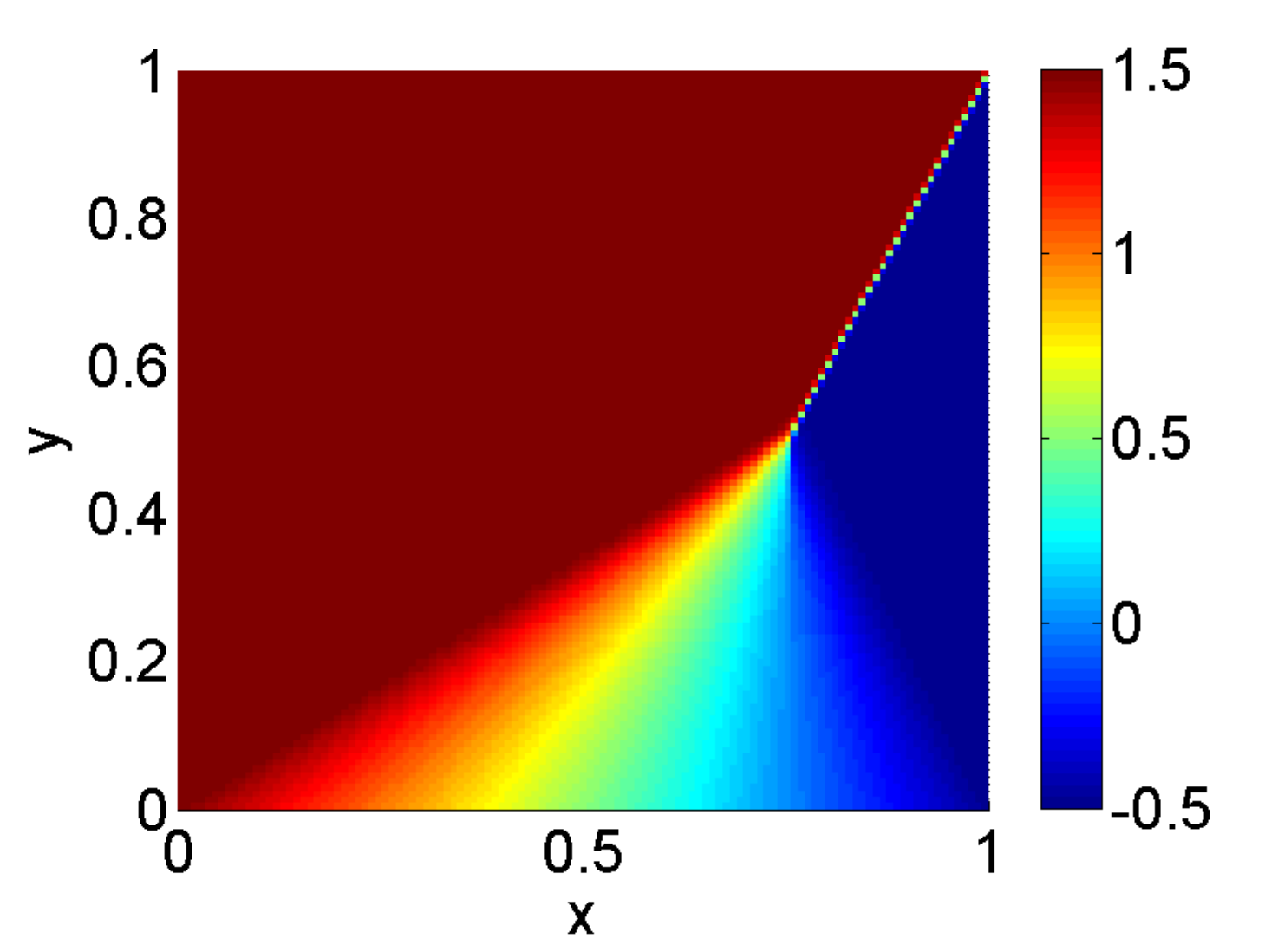}\label{fig:stupidB}}
	\subfigure[]{\includegraphics[width=.48\textwidth]{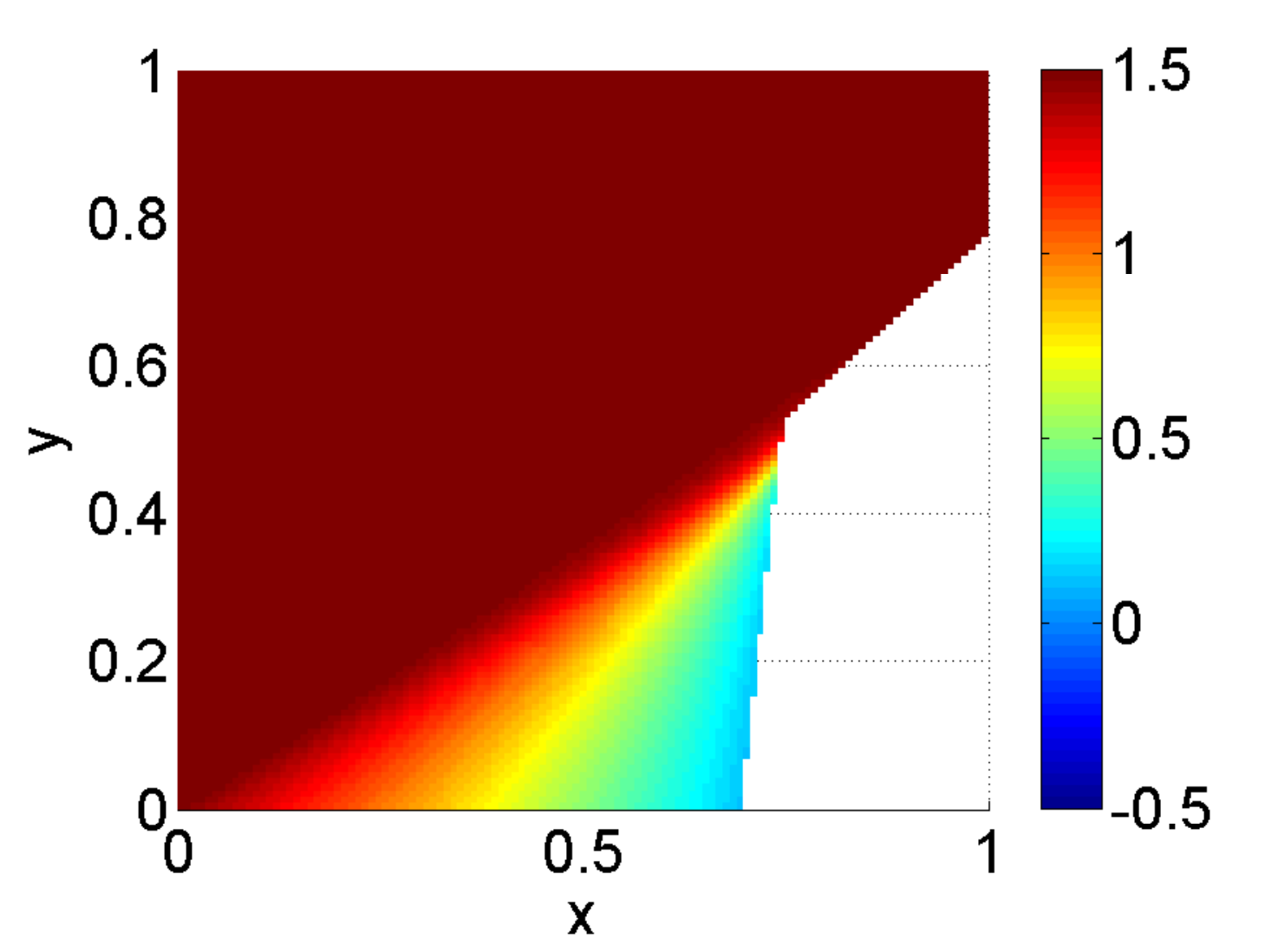}\label{fig:stupidL}}
  	\caption{\subref{fig:stupidB}~Bottom and \subref{fig:stupidL}~left solution branches for the 2D scalar example in~\S\ref{sec:2dscalar} computed on a $120\times120$ grid.}
  	\label{fig:2dscalar}  	
\end{figure}

\subsection{Two-Dimensional Systems}
We can apply the same sweeping procedure to a two-dimensional system
\bq\label{eq:2dsystem} f(U)_x + g(U)_y = a(U,x,y) \eq
with suitable boundary conditions.  For example, to sweep a solution from the left, we would solve the paraxial system
\bq\label{eq:paraxialSystem}
V_x + g\left(f^{-1}(U)\right)_y = a\left(f^{-1}(U),x,y\right),
\eq
using boundary conditions at $x=x_{min}$ as the ``initial condition''.  In this case, boundary conditions at the top and bottom may be specified for some, but not all, of the components of the solution vector.  When permitted by the direction of the characteristics, remaining boundary values at $y = y_{min},y_{max}$ can be updated via upwinding.

\subsubsection{2D Euler equations}\label{sec:euler}
We consider the 2D Euler equations
\bq\label{eq:euler} 
\left(\begin{tabular}{c}$\rho$\\$\rho u$\\$\rho v$\\$E$\end{tabular}\right)_t + \left(\begin{tabular}{c}$\rho u$\\$\rho u^2+p$\\$\rho u v$\\$u(E+p)$\end{tabular}\right)_x  + \left(\begin{tabular}{c}$\rho v$\\$\rho u v$\\$\rho v^2+p$\\$v(E+p)$\end{tabular}\right)_y = 0  
\eq
in the domain
\[ 0 \leq x \leq 4, 0 \leq y \leq 1.\]
Here $p = (\gamma-1)\left(E-\frac{1}{2}\rho(u^2+v^2)\right)$ and $\gamma = 1.4$.  

Following~\cite{Chen_LFSweeping,Shu_WENOHomotopy}, we enforce the boundary conditions
\[
(\rho, u, v, p) = 
\begin{cases}
(1.69997,2.61934,-0.50632, 1.528191) & y = 1\\
(1, 2.9, 0, 1/\gamma) & x = 0.
\end{cases}
\]
A reflection condition (i.e. $v=0$) is imposed at $y=0$ and no boundary conditions are given at $x=4$.

We can actually obtain the entire solution by sweeping once from the left boundary since all the eigenvalues 
of $\nabla f$ ($u-c, u, u, u+c$) are positive throughout.  The computed energy is shown in Figure~\ref{fig:euler}.  Computation times, which are presented in Table~\ref{table:euler}, demonstrate that even for a system, the computational complexity of the sweeping process is linear in the number of grid points.
%We can also look at the eigenvalues of $\nabla g$: $v-c, v, v, v+c$.  On the top boundary, three of the eigenvalues of $\nabla g$ are negative, but the fourth is positive.  This suggests that maybe we should only have three boundary conditions on the top.

\begin{table}[htdp]
\caption{Computation time on an $m\times m$ grid ($N=m^2$) for the 2D Euler equations in~\S\ref{sec:euler}.}
\begin{center}
\begin{tabular}{c||ccccc}
$m$  & 32 & 64 & 128 & 256 & 512 \\
\hline
CPU Time (s) & 3.0 & 7.4 & 25.4 & 92.7 & 354.1\\
\end{tabular}
\end{center}
\label{table:euler}
\end{table}

\begin{figure}[htdp]
	\centering
	{\includegraphics[width=\textwidth,trim=0 50 0 50,clip=true]{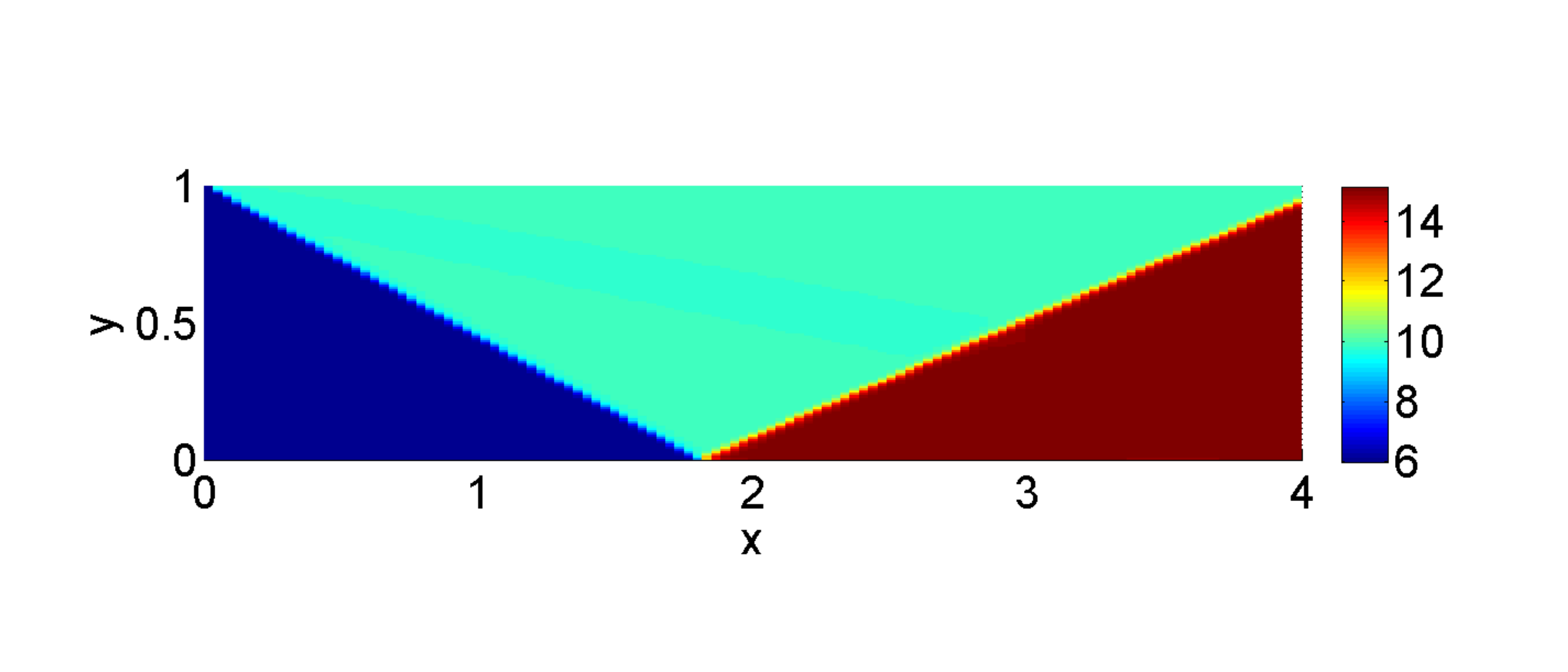}}
  	\caption{Computed energy for the 2D Euler shock reflection problem in~\S\ref{sec:euler} on a $120\times120$ grid.}
  	\label{fig:euler}  	
\end{figure}

\section{Conclusions}
In this article, we have introduced a fast sweeping approach for computing steady state solutions to systems of conservation laws.  Two of the biggest assets of this approach are its computational efficiency and ability to capture shocks sharply.  The methods can also be combined with the numerical flux of choice, can be used to solve problems with multiple steady states, and can solve problems that involve different types of boundary conditions.

\bibliographystyle{plain}
\bibliography{SweepingBib}

\end{document}